\documentclass[twocolumn,letterpaper]{IEEEAerospaceCLS}  

\usepackage[]{graphicx}    
\usepackage[colorlinks=true, urlcolor=blue, linkcolor=black, citecolor=black]{hyperref} 
\newcommand{\ignore}[1]{}  
\usepackage{amsmath}
\usepackage{amssymb}
\usepackage{bm}
\usepackage{mathtools}
\usepackage{xcolor}
\usepackage{cleveref}
\usepackage[ruled, lined, linesnumbered, commentsnumbered]{algorithm2e}
\usepackage{float}
\usepackage{enumitem}
\usepackage{xurl} 
\usepackage{hyperref} 

\newtheorem{theorem}{Theorem}[section]
\newtheorem{lemma}[theorem]{Lemma}

\newtheorem{assumption}{Assumption}
\newtheorem{remark}{Remark}[section]

\begin{document}
\title{Optimality-Informed Neural Networks for Lunar Landing Trajectory Optimization}

\author{%
Zhenbo Wang\\
Department of Mechanical and Aerospace Engineering\\
The University of Tennessee\\
1512 Middle Drive, Knoxville, TN 37996\\
zwang124@utk.edu
\thanks{\footnotesize 979-8-3195-1048-8/27/$\$31.00$ \copyright2027 IEEE}              
}

\maketitle

\thispagestyle{plain}
\pagestyle{plain}

\maketitle

\thispagestyle{plain}
\pagestyle{plain}

\begin{abstract}
This paper develops an Optimality-Informed Neural Network (OINN) approach for the energy-optimal, free-final-time powered descent of a lunar lander from any initial position, velocity, and mass within a bounded operating envelope to a fixed landing site with zero terminal velocity. Building on a recent framework that jointly embeds Pontryagin's minimum principle and the Hamilton-Jacobi-Bellman equation for general nonlinear optimal control, the proposed OINN approach specializes that idea to a lunar landing problem with free time of flight and fixed terminal state. Every boundary and transversality condition is hard-encoded into the network architecture by construction, the closed-form Pontryagin-optimal thrust magnitude and direction law is substituted directly rather than learned, and the remaining state, costate, and an auxiliary value-function output are trained against a physics-residual loss formed entirely from the necessary conditions of optimality, with no precomputed optimal trajectories required. A preliminary theoretical analysis is explored, establishing a stochastic-optimization stationarity guarantee for the offline training procedure, an explicit bound translating the achieved training residual into bounds on touchdown position, touchdown velocity, and flight-time error, and a fixed, input-independent onboard computational and memory cost suitable for real-time deployment. Numerical simulations evaluate the trained policy, with no retraining, against an independently solved indirect-method boundary-value problem at six representative initial states spanning the operating envelope and against eighty additional Monte Carlo simulation runs, demonstrating close agreement with the indirect-method solution and consistently small dynamics and transversality residuals throughout the envelope.
\end{abstract}

\tableofcontents

\section{Introduction}
\label{sec:introduction}

The Apollo program first demonstrated that a spacecraft can be guided through a propulsive descent from lunar orbit to a safe touchdown using onboard computation, and the past several years have seen a renewed, broadened effort to repeat and extend that capability. NASA's Artemis campaign has carried a crew beyond low Earth orbit for the first time since Apollo and is developing crewed lunar landers, while NASA's Commercial Lunar Payload Services initiative and international programs such as China's Chang'e series and India's Chandrayaan series have flown, or are planning, a steadily increasing cadence of robotic landings \cite{wang2025a}. Several recent commercial landers have also tipped over or landed with reduced margin during the final powered-descent phase, underscoring that translating a precomputed reference trajectory into a propellant-efficient, accurate, and robust touchdown remains an open engineering challenge rather than a solved problem \cite{wang2025b}. Because the powered-descent phase is short, safety-critical, and necessarily executed with no opportunity for ground intervention, the guidance law that generates thrust and pointing commands in real time, from whatever position, velocity, and mass the vehicle actually has at the start of descent rather than from one trajectory planned in advance, is as consequential to mission success as the trajectory-design process that informs it. Trajectory optimization, in the sense of finding a propellant- or time-optimal descent profile subject to the vehicle's dynamics and constraints, is therefore not merely a planning convenience but a central engineering requirement for any lander that must reach a prescribed site with a prescribed terminal velocity while using as little of its onboard propellant as possible \cite{liu2026survey}.

The theoretical foundation for solving such problems dates to Pontryagin's minimum principle, which characterizes an optimal trajectory through a two-point boundary-value problem in the system's states and costates. Indirect methods solve this boundary-value problem directly, typically by a shooting or collocation procedure, and can in principle recover a globally optimal solution to high numerical precision. In practice, however, they are sensitive to the initial guess supplied for the unknown costates, whose physical units and magnitudes generally differ greatly from those of the states, and convergence for a new initial condition is not guaranteed, so each new powered-descent scenario can require a fresh, potentially fragile numerical solve~\cite{lu2018propellant}. Direct methods avoid this sensitivity by discretizing the trajectory and the control history and transcribing the optimal control problem into a finite-dimensional nonlinear program, solved by general-purpose nonlinear or pseudospectral optimization software~\cite{ross2012review}. This transcription is more robust to a poor initial guess, but the resulting nonlinear program can still require a variable, a priori unknown number of iterations to converge, and is correspondingly difficult to certify for execution within a fixed onboard computation budget. Reformulating the powered-descent problem as a convex program, most notably through the lossless-convexification result that recasts the nonconvex thrust-bound and direction constraints exactly as convex ones for a class of powered-descent problems, removed much of this uncertainty by guaranteeing convergence to the optimum in a bounded number of iterations of an interior-point solver. Convex-optimization-based guidance of this kind has since flown, or been proposed for flight, on several planetary and lunar landing missions~\cite{acikmese2007convex,blackmore2010minimum}. Even so, an exact convex reformulation is only available for problems with a specific, favorable structure. More general nonlinear dynamics, path constraints, or cost functions typically require successive convexification or other iterative approximations that reintroduce the variable iteration counts and convergence sensitivity that lossless convexification was designed to avoid~\cite{wang2024survey}.

An alternative line of work replaces the online numerical solve entirely with an offline-trained neural network that maps the current state directly to a control command, so that only a single, fast forward pass is required onboard. Supervised, imitation-learning approaches train such a network on a large dataset of optimal trajectories precomputed offline by an indirect or direct method, and have demonstrated real-time, closed-loop guidance for planetary landing and other guidance problems~\cite{sanchez2018real,cheng2019real}. Reinforcement learning instead trains a policy through repeated simulated interaction with the vehicle's dynamics, optimizing a reward signal rather than imitating precomputed examples, and has been applied directly to fuel-optimal lunar and planetary powered descent~\cite{scorsoglio2022image,scorsoglio2025meta}. Both families of methods share two limitations that motivate the approach developed in this paper.
First, both are data- or interaction-hungry. Imitation learning requires a training dataset whose generation cost and coverage scale with the dimension and extent of the operating envelope to be served, while reinforcement learning requires extensive trial-and-error exploration that can itself violate safety-critical state or control constraints before the policy has converged.
Second, neither approach enforces the governing dynamics or the necessary conditions of optimality directly. The underlying physical laws of optimal control enter only implicitly, through the training examples or the reward signal, so the deployed network offers no formal guarantee of dynamic feasibility, constraint satisfaction, or proximity to optimality once it is queried at a state not well represented in training.

Physics-informed neural networks (PINN), which embed the residual of a governing differential equation directly into the training loss alongside or instead of labeled data, offer a third alternative that trains the network against the physical laws of optimal control themselves rather than against either precomputed examples or a reward signal~\cite{raissi2019physics}. Several recent studies have begun applying this idea to aerospace guidance and control by embedding either Pontryagin's minimum principle or the Hamilton-Jacobi-Bellman equation into a physics-informed network~\cite{schiassi2022physics,d2025physics,na2024physics}, with encouraging improvements in physical consistency over purely data-driven networks. However, these studies typically embed only one of the two necessary-condition perspectives rather than their full joint structure and are often tailored to a single, fixed initial condition rather than a genuine closed-loop policy valid over a whole family of initial states. Also, because the Hamilton-Jacobi-Bellman equation can admit multiple solutions for a general nonlinear problem, existing research still relies on at least some quantity of precomputed optimal data to anchor the network to the correct one~\cite{misztela2018nonuniqueness}.

This paper builds on a recently proposed optimality-principles-informed neural network framework that addresses the above limitations for general nonlinear, infinite-horizon optimal regulation problems by unifying the full set of necessary conditions implied jointly by Pontryagin's minimum principle and the Hamilton-Jacobi-Bellman equation within a single architecture, achieving high sample efficiency, including training with little or even zero precomputed optimal data~\cite{wang2026opinn}. The present paper specializes and extends that framework, here renamed the Optimality-Informed Neural Network (OINN) approach, to the free-final-time, fixed-terminal-boundary structure of energy-optimal lunar powered descent, a structure not covered by the original infinite-horizon regulation formulation. The architecture hard-encodes the lunar lander's specific boundary and transversality conditions, including a free final time bounded only from below, exploits the closed-form constant and affine costate structure that follows from this particular Hamiltonian to remove most of the network's representational burden, substitutes the resulting Pontryagin-optimal thrust and direction law in closed form rather than learning it, and trains only against the necessary conditions of optimality, with the value function retained solely as an auxiliary Bellman-consistency diagnostic rather than a full Hamilton-Jacobi-Bellman embedding. The resulting policy is trained entirely offline, requires only a fixed, small number of floating-point operations to evaluate online, and is accompanied by an explicit theoretical analysis that translates training convergence and the achieved residual magnitude into computable bounds on touchdown error, flight-time error, and onboard computational cost.

The remainder of this paper is organized as follows. Section~\ref{sec:problem_formulation} formulates the free-final-time, energy-optimal lunar powered-descent problem over a bounded operating envelope of initial states. Section~\ref{sec:optimality_conditions} develops the necessary conditions of optimality for this problem via Pontryagin's minimum principle, including the closed-form optimal control law and the transversality conditions for the free terminal mass and free final time. Section~\ref{sec:OINN} presents the OINN architecture and the physics-residual training loss that embeds these conditions. Section~\ref{sec:theoretical_analysis} establishes the theoretical convergence, accuracy, and onboard computational properties of the trained policy. Section~\ref{sec:numerical_simulations} reports numerical simulations of a representative lunar landing scenario, comparing the trained policy against an independently solved indirect-method solution and a Monte Carlo evaluation across the operating envelope. Section~\ref{sec:conclusions} concludes the paper.

\section{Problem Formulation}
\label{sec:problem_formulation}


This paper considers a lunar lander descending under rocket propulsion from an initial position, velocity, and mass to a prescribed landing site with a prescribed (typically zero) terminal velocity. The lander carries no aerodynamic surfaces and is controlled entirely by a single throttleable, gimbaled main engine, whose thrust magnitude and pointing direction constitute the control input.

Motion of the lander is described in a local-vertical, local-horizontal (LVLH) frame with East--North--Up (ENU) coordinates, with its origin fixed at the designated landing site. The lunar gravitational acceleration is treated as constant over the relatively short, low-altitude descent and is directed along the negative ``Up'' axis. Unlike a single offline trajectory computed for one fixed starting condition, the descent considered here is intended to start from any initial position, velocity, and mass within a bounded operational range, since the lander's actual state at the moment powered descent begins is not known precisely in advance.


Let $\bm{r} = [x,\,y,\,z]^{\top} \in \mathbb{R}^{3}$ denote the lander position, $\bm{v} = [v_{x},\,v_{y},\,v_{z}]^{\top} \in \mathbb{R}^{3}$ its velocity, and $m \in \mathbb{R}$ its instantaneous mass. The control is the thrust magnitude $T$ together with a unit vector $\hat{\bm{\imath}}_{\theta} \in \mathbb{R}^{3}$ providing the thrust direction. The equations of motion are \cite{cheng2019real}
\begin{align}
\dot{\bm{r}} &= \bm{v}, \label{eq:eom_r} \\
\dot{\bm{v}} &= \frac{T}{m}\,\hat{\bm{\imath}}_{\theta} + \bm{g}, \label{eq:eom_v} \\
\dot{m} &= -\,\frac{T}{I_{sp}\,g_{0}}, \label{eq:eom_m}
\end{align}
where $\bm{g} = [0,\,0,\,-g_{\mathrm{moon}}]^{\top}$ is the constant lunar gravitational acceleration vector, $I_{sp}$ is the engine specific impulse, and $g_{0}$ is the standard Earth sea-level gravitational acceleration. Equation~\eqref{eq:eom_m} is the usual rocket-equation mass-depletion law: propellant is consumed at a rate proportional to thrust.


Collecting the above into state and control vectors,
\begin{equation}
\bm{x} = \begin{bmatrix} x \\ y \\ z \\ v_{x} \\ v_{y} \\ v_{z} \\ m \end{bmatrix} \in \mathbb{R}^{7}, \qquad
\bm{u} = \begin{bmatrix} T \\ \hat{\bm{\imath}}_{\theta} \end{bmatrix} \in \mathbb{R}^{4},
\label{eq:state_control}
\end{equation}
the dynamics \eqref{eq:eom_r}--\eqref{eq:eom_m} can be written compactly as $\dot{\bm{x}} = \bm{f}(\bm{x},\bm{u})$. The state--control pair is nonlinear in two distinct ways that matter throughout this paper: the velocity dynamics contain the product of $T/m$, a ratio of two states/controls, with the direction vector, and the direction vector itself is constrained to the surface of the unit sphere, a nonconvex quadratic equality constraint.


Rather than designing the guidance law around a single, precisely known initial condition, this paper targets a feedback policy, where the lander's actual position, velocity, and mass at the moment powered descent begins are taken to be free to vary, within a bounded region of plausible initial states, rather than fixed to a particular condition. Define the region of free initial states
\begin{align}\label{eq:region_omega}
\Omega = \Big\{ (\bm{r}_{0},\bm{v}_{0},m_{0}) :\ &\bm{r}_{0} \in \big[\bm{r}_{0}^{\mathrm{lo}},\bm{r}_{0}^{\mathrm{hi}}\big], \nonumber \\
&\bm{v}_{0} \in \big[\bm{v}_{0}^{\mathrm{lo}},\bm{v}_{0}^{\mathrm{hi}}\big], \nonumber \\
&m_{0} \in \big[m_{0}^{\mathrm{lo}},m_{0}^{\mathrm{hi}}\big] \Big\},
\end{align}
where $\bm{r}_{0}^{\mathrm{lo}}, \bm{r}_{0}^{\mathrm{hi}} \in \mathbb{R}^{3}$ and $\bm{v}_{0}^{\mathrm{lo}}, \bm{v}_{0}^{\mathrm{hi}} \in \mathbb{R}^{3}$ denote component-wise lower and upper bounds on the initial position and velocity, and $m_{0}^{\mathrm{lo}}, m_{0}^{\mathrm{hi}} \in \mathbb{R}$ denote scalar lower and upper bounds on the initial mass. Section~\ref{sec:numerical_simulations} gives the specific bounds used in numerical simulations.

\begin{remark}
The vehicle dry mass $m_{\mathrm{dry}}$ is held fixed across $\Omega$ rather than scaled with $m_{0}$. This is a simplifying choice. A more detailed vehicle model might let the dry mass scale with vehicle size. As a consequence, the available propellant fraction $(m_{0}-m_{\mathrm{dry}})/m_{0}$ varies somewhat across $\Omega$ depending on where $m_{0}$ falls within its range.
\end{remark}


The trajectory begins at an initial state $(\bm{r}_{0},\bm{v}_{0},m_{0}) \in \Omega$ at the initial time $t_{0}$ and must reach a prescribed landing site at the final time $t_{f}$:
\begin{align}
\bm{r}(t_{0}) = \bm{r}_{0}, \quad \bm{v}(t_{0}) = \bm{v}_{0}, \quad m(t_{0}) = m_{0}, \label{eq:bc_initial} \\
\bm{r}(t_{f}) = \bm{r}_{f}, \quad \bm{v}(t_{f}) = \bm{v}_{f}, \label{eq:bc_terminal}
\end{align}
where $\bm{r}_{f}$ and $\bm{v}_{f}$ are fixed and identical for every $(\bm{r}_{0},\bm{v}_{0},m_{0}) \in \Omega$. In this paper, $\bm{r}_{f} = \bm{0}$ and $\bm{v}_{f} = \bm{0}$, corresponding to a soft landing exactly at the designated site. The terminal mass $m(t_{f})$ is not prescribed; it is free to be determined by the optimization, and less propellant used is preferable. This combination of free initial conditions together with a fixed terminal position and velocity but a free terminal mass is the central structural feature that the OINN architecture in Section~\ref{sec:OINN} is built to satisfy.


Beyond the boundary conditions, the trajectory must respect
\begin{align}
0 \leq T(t) \leq T_{\max}, \qquad \big\| \hat{\bm{\imath}}_{\theta}(t) \big\|_{2}^{2} = 1, \label{eq:constraints_T_dir} \\
m_{\mathrm{dry}} \leq m(t) \leq m_{0}, \label{eq:constraints_mass}
\end{align}
i.e., the thrust magnitude is bounded by the engine's maximum thrust $T_{\max}$, the thrust direction is constrained to be a unit vector, and the mass must remain between the dry mass $m_{\mathrm{dry}}$, comprising structure, engine, and payload, and the initial wet mass $m_{0}$.


A further departure from a conventional fixed-horizon formulation is that the time of flight is treated as a decision variable here rather than a parameter chosen in advance. The final time is left free,
\begin{equation}
t_{f} > t_{0},
\label{eq:tf_free}
\end{equation}
to be determined by the optimization for each initial state, rather than fixed to a single mission-wide value. Letting the final time be free is appealing at the trajectory-design stage, since it allows the optimization itself to discover the most propellant-efficient flight duration for each initial state, rather than committing to a flight duration chosen by other means. Section~\ref{subsec:predicted_tf} introduces a nominal minimum flight time, used purely as a numerical safeguard inside the neural-network architecture rather than as a constraint of the problem itself.


An energy-optimal landing is sought, with a cost functional that combines the usual energy-like thrust penalty with a constant time penalty $\rho \geq 0$:
\begin{equation}
J = \int_{t_{0}}^{t_{f}} \Big( T^{2}(t) + \rho \Big)\, dt.
\label{eq:objective}
\end{equation}
\begin{remark}
\label{rem:time_penalty}
The time penalty $\rho$ is necessary in \eqref{eq:objective} to admit a finite minimizer when the final time is free. For fixed boundary conditions, lengthening $t_{f}$ always permits a gentler deceleration profile, so $\int T^{2}\,dt$ alone decreases monotonically as $t_{f}$ grows without bound. Without a competing term, the problem will have no finite optimal $t_{f}$ at all. The infimum of $\int T^{2}\,dt$ over $t_{f} > t_{0}$ is approached only in the limit $t_{f} \to \infty$. The penalty $\rho$, having the same units as $T^{2}$, restores a genuine fuel-versus-duration trade-off with a well-defined finite optimum, and is exactly the type of term a real mission planner would weigh: additional flight time is not free even when it reduces propellant consumption.
\end{remark}


For any single fixed $(\bm{r}_{0},\bm{v}_{0},m_{0}) \in \Omega$, the complete three-dimensional, energy-optimal lunar landing trajectory optimization problem with a free final time is
\begin{equation}
\min_{\bm{u}(\cdot),\, t_{f}} \ J = \int_{t_{0}}^{t_{f}} \Big( T^{2}(t) + \rho \Big)\, dt
\label{eq:problem1}
\end{equation}
subject to the dynamics \eqref{eq:eom_r}--\eqref{eq:eom_m}, the boundary conditions \eqref{eq:bc_initial}--\eqref{eq:bc_terminal}, the control and mass constraints \eqref{eq:constraints_T_dir}--\eqref{eq:constraints_mass}, and the free-final-time condition \eqref{eq:tf_free}. We denote it as Problem~1.

Strictly speaking, because $(\bm{r}_{0},\bm{v}_{0},m_{0})$ is permitted to be any point of $\Omega$ rather than one fixed triple, Problem~1 is a family of optimal control problems, one for every $(\bm{r}_{0},\bm{v}_{0},m_{0}) \in \Omega$. A classical indirect or direct transcription method must be re-solved from scratch for each member of this family. The goal of this paper is instead a single guidance law, a function of the current time-to-go and the specific initial state being flown, that solves every member of the family at once:
\begin{align}
&\bm{u}^{\ast}\big(t;\,\bm{r}_{0},\bm{v}_{0},m_{0}\big), \quad t_{f}^{\ast}\big(\bm{r}_{0},\bm{v}_{0},m_{0}\big), \nonumber \\
&\hspace{2.3em} \forall\, (\bm{r}_{0},\bm{v}_{0},m_{0}) \in \Omega.
\label{eq:policy_goal}
\end{align}
Problem~1 is nonconvex for two reasons: the velocity dynamics~\eqref{eq:eom_v} contain a nonlinear $T \, \hat{\bm{\imath}}_{\theta}/m$ coupling term, and the direction constraint in~\eqref{eq:constraints_T_dir} is a nonconvex quadratic equality, the surface rather than the interior of the unit ball. Section~\ref{sec:optimality_conditions} develops the necessary conditions of optimality for Problem~1 via Pontryagin's minimum principle; Section~\ref{sec:OINN} then develops a neural-network architecture, trained by domain randomization over $\Omega$, that approximates the feedback law~\eqref{eq:policy_goal} directly.

\section{Optimality Conditions}
\label{sec:optimality_conditions}

This section develops the necessary conditions of optimality for Problem~1 posed in Section~\ref{sec:problem_formulation} via Pontryagin's minimum principle \cite{kirk2004optimal}. These conditions are used twice in the remainder of this paper: once to obtain a closed-form expression for the optimal control that is substituted directly, rather than learned, inside the neural-network architecture of Section~\ref{sec:OINN}, and once to define the costate-related quantities that the same architecture represents.

\subsection{Hamiltonian and Costate Dynamics}
\label{subsec:hamiltonian}

Introduce costate vectors $\bm{\lambda}_{r}, \bm{\lambda}_{v} \in \mathbb{R}^{3}$ and a scalar costate $\lambda_{m} \in \mathbb{R}$, associated with $\bm{r}$, $\bm{v}$, and $m$ respectively. The Hamiltonian for Problem~1 is
\begin{equation}
H = \bm{\lambda}_{r}^{\top} \bm{v} + \bm{\lambda}_{v}^{\top}\Big( \frac{T}{m}\,\hat{\bm{\imath}}_{\theta} + \bm{g} \Big) - \lambda_{m}\,\frac{T}{I_{sp}\,g_{0}} + T^{2} + \rho.
\label{eq:hamiltonian}
\end{equation}
Applying the costate equations $\dot{\bm{\lambda}} = -\,\partial H/\partial \bm{x}$ gives
\begin{align}
\dot{\bm{\lambda}}_{r} &= \bm{0}, \label{eq:costate_r} \\
\dot{\bm{\lambda}}_{v} &= -\,\bm{\lambda}_{r}, \label{eq:costate_v} \\
\dot{\lambda}_{m} &= \frac{T}{m^{2}}\, \bm{\lambda}_{v}^{\top} \hat{\bm{\imath}}_{\theta}. \label{eq:costate_m}
\end{align}
Two structural simplifications follow directly from~\eqref{eq:costate_r}--\eqref{eq:costate_m} and are exploited explicitly in the network architecture of Section~\ref{sec:OINN}: the position costate $\bm{\lambda}_{r}$ is exactly constant along the optimal trajectory, and the velocity costate $\bm{\lambda}_{v}$ is therefore an exactly affine (linear) function of time. Only the mass-costate equation~\eqref{eq:costate_m} is genuinely nonlinear, through its coupling to $T$, $m$, and $\hat{\bm{\imath}}_{\theta}$.

\subsection{Optimal Control Law}
\label{subsec:control_law}

Pontryagin's minimum principle states that the optimal control minimizes the Hamiltonian pointwise over the admissible control set. Minimizing~\eqref{eq:hamiltonian} over the unit-vector direction $\hat{\bm{\imath}}_{\theta}$, for $T \geq 0$, is a linear minimization over the unit sphere, whose solution points the thrust exactly opposite the velocity costate:
\begin{equation}
\hat{\bm{\imath}}_{\theta}^{\ast} = -\,\frac{\bm{\lambda}_{v}}{\big\| \bm{\lambda}_{v} \big\|_{2}}.
\label{eq:optimal_direction}
\end{equation}
Substituting~\eqref{eq:optimal_direction} back into~\eqref{eq:hamiltonian} leaves an unconstrained scalar minimization over $T$ of a convex quadratic, clipped to the admissible range $[0,T_{\max}]$. The result is
\begin{equation}
T^{\ast} = \mathrm{clip}\!\left( \frac{1}{2}\left( \frac{\big\| \bm{\lambda}_{v} \big\|_{2}}{m} + \frac{\lambda_{m}}{I_{sp}\,g_{0}} \right),\ 0,\ T_{\max} \right),
\label{eq:optimal_thrust}
\end{equation}
where $\mathrm{clip}(\cdot,0,T_{\max})$ denotes clamping to the interval $[0,T_{\max}]$. Substituting the optimal direction~\eqref{eq:optimal_direction} into the mass-costate dynamics~\eqref{eq:costate_m} also gives a convenient closed form for use along the optimal trajectory:
\begin{equation}
\dot{\lambda}_{m} = -\,\frac{T^{\ast}\, \big\| \bm{\lambda}_{v} \big\|_{2}}{m^{2}}.
\label{eq:lambdam_optimal}
\end{equation}
\begin{remark}
The time penalty $\rho$ appears additively in the Hamiltonian~\eqref{eq:hamiltonian} but does not depend on the control. Consequently, $\rho$ has no effect whatsoever on the minimizing control law~\eqref{eq:optimal_direction}--\eqref{eq:optimal_thrust}; it affects only the transversality condition that determines the optimal final time derived below.
\end{remark}

\subsection{Transversality Condition for the Free Terminal Mass}
\label{subsec:transversality_mass}

Because the terminal mass $m(t_{f})$ is free and no terminal cost is associated with it, the standard transversality condition requires the corresponding costate to vanish at the final time:
\begin{equation}
\lambda_{m}(t_{f}) = 0.
\label{eq:transversality_mass}
\end{equation}
No analogous condition is needed for $\bm{\lambda}_{r}$ or $\bm{\lambda}_{v}$: the transversality argument runs the opposite way for those, requiring a costate's terminal value to vanish only when the corresponding state's terminal value is free. Here $\bm{r}(t_{f})$ and $\bm{v}(t_{f})$ are fixed boundary conditions, so $\bm{\lambda}_{r}(t_{f})$ and $\bm{\lambda}_{v}(t_{f})$ are left unconstrained by transversality; instead, the fixed terminal position and velocity become two additional algebraic conditions that the solution must satisfy directly on the state trajectory.

\subsection{Transversality Condition for the Free Final Time}
\label{subsec:transversality_tf}

When the final time is free, with no explicit terminal cost depending on $t_{f}$ beyond the state constraints already imposed, the standard transversality condition is $H(t_{f}) = 0$. Because the dynamics~\eqref{eq:eom_r}--\eqref{eq:eom_m} and the integrand of~\eqref{eq:objective} are both autonomous, that is, neither depends explicitly on $t$, the Hamiltonian is constant along the optimal trajectory, $dH/dt = 0$, so
\begin{equation}
H(t) = 0, \qquad \forall\, t \in [t_{0},t_{f}],
\label{eq:H_identically_zero}
\end{equation}
not merely at the final time. Substituting the optimal control law~\eqref{eq:optimal_direction}--\eqref{eq:optimal_thrust} into the Hamiltonian~\eqref{eq:hamiltonian} and simplifying, using $\hat{\bm{\imath}}_{\theta}^{\ast\top}\bm{\lambda}_{v} = -\big\|\bm{\lambda}_{v}\big\|_{2}$ and $T^{\ast}$ from~\eqref{eq:optimal_thrust} in the unsaturated regime, gives the compact closed form
\begin{equation}
H^{\ast}(t) = \bm{\lambda}_{r}^{\top}\bm{v} + \bm{\lambda}_{v}^{\top}\bm{g} - \big(T^{\ast}\big)^{2} + \rho.
\label{eq:H_closed_form}
\end{equation}
Condition~\eqref{eq:H_identically_zero} together with~\eqref{eq:H_closed_form} is what determines the final time when it is free: for a given initial state, $t_{f}^{\ast}$ is the value of the final time for which a solution of the costate dynamics~\eqref{eq:costate_r}--\eqref{eq:costate_m}, the optimal control law~\eqref{eq:optimal_direction}--\eqref{eq:optimal_thrust}, and the boundary conditions~\eqref{eq:bc_initial}--\eqref{eq:bc_terminal} also satisfies~\eqref{eq:H_identically_zero}--\eqref{eq:H_closed_form}.

\begin{remark}
\label{rem:tf_interior}
Condition~\eqref{eq:H_identically_zero} is the only information available about $t_{f}^{\ast}$. Nothing in Problem~1 bounds the final time from above, and Remark~\ref{rem:time_penalty} already established that the time penalty $\rho$ is what gives this condition a finite root in the first place. The OINN developed in Section~\ref{sec:OINN} predicts a final time for each initial state directly from this condition, subject only to a nominal positivity safeguard, $t_{f}^{\ast} > t_{f}^{\min}$ for a small fixed $t_{f}^{\min} > 0$, imposed for numerical reasons inside the network architecture rather than as a feature of Problem~1 itself.
\end{remark}

\subsection{Two-Point Boundary-Value Problem Summary}
\label{subsec:tpbvp_summary}

Collecting the results above, the indirect-method solution of Problem~1, for any one fixed $(\bm{r}_{0},\bm{v}_{0},m_{0}) \in \Omega$, reduces to a two-point boundary-value problem in fourteen scalar states and costates, $\bm{x},\bm{\lambda} \in \mathbb{R}^{7}$ each, together with one unknown scalar parameter $t_{f}$, with the optimal control~\eqref{eq:optimal_direction}--\eqref{eq:optimal_thrust} substituted into the dynamics. The associated boundary and transversality conditions are the seven initial conditions~\eqref{eq:bc_initial}, the six fixed terminal conditions on $\bm{r}$ and $\bm{v}$ from~\eqref{eq:bc_terminal}, the transversality condition~\eqref{eq:transversality_mass} for the free terminal mass, and the transversality condition~\eqref{eq:H_identically_zero} for the free final time, fifteen scalar conditions in total, matching the fourteen unknown states and costates plus the one unknown parameter $t_{f}$, as required for a well-posed parametric boundary-value problem. This is the boundary-value problem solved independently as a numerical check on the OINN in Section~\ref{sec:numerical_simulations}. It is also the structure that the network architecture in Section~\ref{sec:OINN} is built to satisfy through its architecture and training loss, simultaneously for every $(\bm{r}_{0},\bm{v}_{0},m_{0}) \in \Omega$, rather than through explicit numerical shooting for one fixed instance.

\section{Optimality-Informed Neural Network Approach}
\label{sec:OINN}

Motivated by \cite{wang2026opinn}, this section develops the optimality-informed neural network (OINN) approach to approximating the feedback law~\eqref{eq:policy_goal} directly, rather than solving Problem~1 independently for each $(\bm{r}_{0},\bm{v}_{0},m_{0}) \in \Omega$. The defining feature of the OINN approach is that the optimality conditions derived in Section~\ref{sec:optimality_conditions}, rather than examples of optimal trajectories, supply the training signal. The network is trained to satisfy the Hamiltonian, the costate dynamics, and the transversality conditions everywhere in $\Omega$, with the closed-form control law~\eqref{eq:optimal_direction}--\eqref{eq:optimal_thrust} substituted directly.

\subsection{Network Architecture}
\label{subsec:architecture}

Figure~\ref{fig:architecture} shows the overall architecture. Two sub-networks share the task of representing the policy: a \emph{condition encoder} that sees only the initial state $(\bm{r}_{0},\bm{v}_{0},m_{0})$, never the time variable, and a \emph{main trunk} that sees the time variable together with the initial state. This split follows directly from the structural simplification identified in Section~\ref{subsec:hamiltonian}: because $\bm{\lambda}_{r}$ is exactly constant and $\bm{\lambda}_{v}$ is therefore an exactly affine function of time for any initial state, only the two constant parameters of that affine function depend on $(\bm{r}_{0},\bm{v}_{0},m_{0})$ while their functional form in $t$ never changes. It is therefore both correct and more efficient to compute those two parameters once per initial state outside of any per-time-step computation, rather than re-deriving them inside a network that also processes $t$.

\begin{figure}
\centering
\includegraphics[width=\columnwidth]{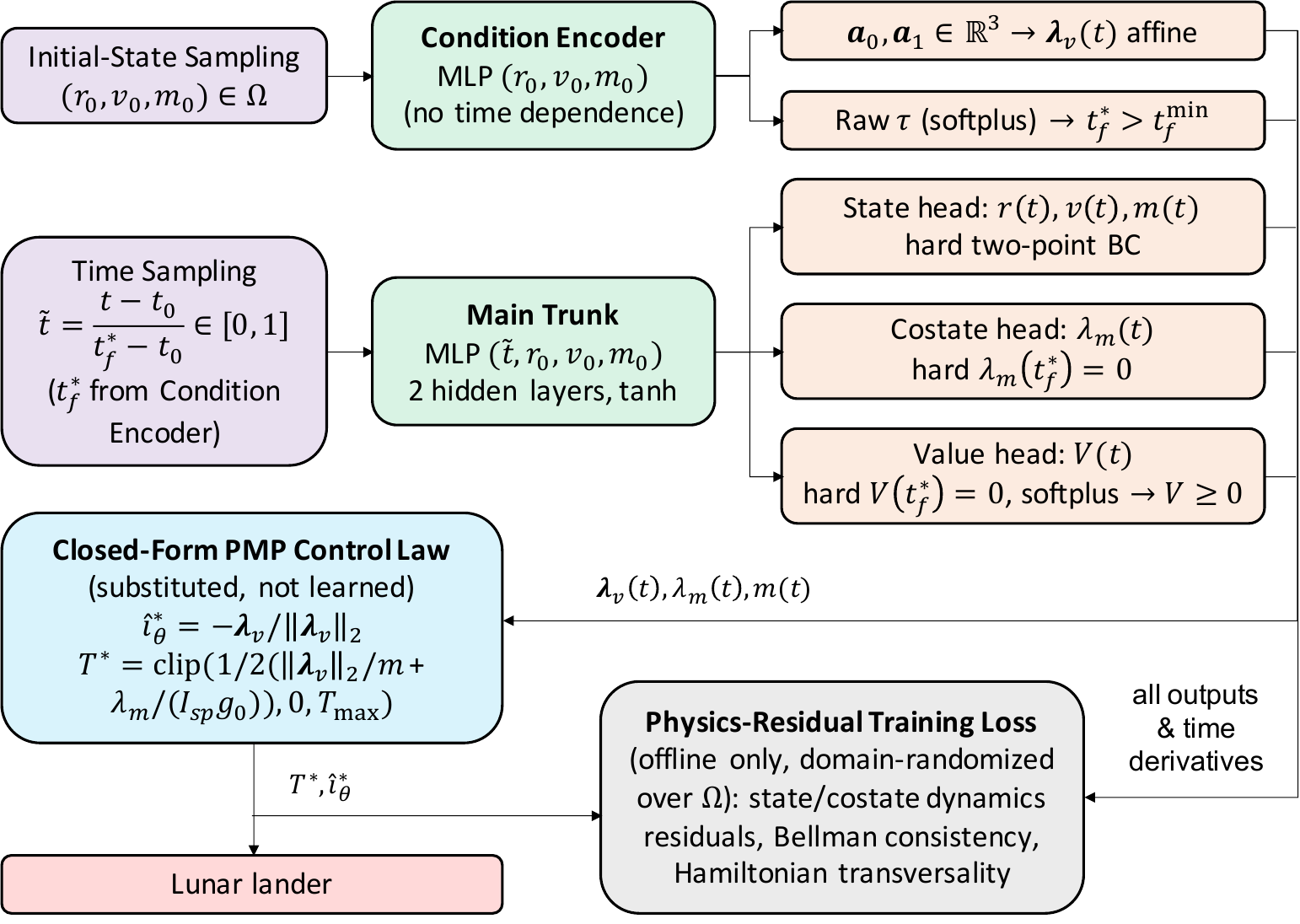}\\
\caption{OINN architecture. A condition encoder, seeing only the initial state, produces the affine velocity-costate parameters and the predicted final time; a main trunk, seeing time together with the initial state, produces the state, mass-costate, and value bubble corrections. All boundary and transversality conditions are satisfied exactly by construction for any initial state; the PMP-optimal control law is substituted in closed form rather than learned.}
\label{fig:architecture}
\end{figure}

The condition encoder's input is the normalized initial state alone. Writing $\bm{r}_{0,\mathrm{norm}}$, $\bm{v}_{0,\mathrm{norm}}$, and $m_{0,\mathrm{norm}}$ for the components of $(\bm{r}_{0},\bm{v}_{0},m_{0})$ rescaled by the center and half-width of their respective ranges in $\Omega$, so that the whole region maps to approximately $[-1,1]$ in every coordinate, the condition encoder computes
\begin{equation}
\big( \bm{a}_{0},\,\bm{a}_{1},\,\tau \big) = \mathrm{NN}_{\mathrm{cond}}\big( \bm{r}_{0,\mathrm{norm}},\,\bm{v}_{0,\mathrm{norm}},\,m_{0,\mathrm{norm}} \big),
\label{eq:cond_encoder}
\end{equation}
where $\bm{a}_{0},\bm{a}_{1} \in \mathbb{R}^{3}$ parameterize the affine velocity costate, defined in~\eqref{eq:bubble_lambda_v} below, and $\tau \in \mathbb{R}$ is a single scalar raw output used to construct the predicted final time. The main trunk's input additionally includes the normalized time $\tilde{t} = (t-t_{0})/(t_{f}^{\ast}-t_{0})$, so that $\tilde{t} \in [0,1]$ regardless of the value of $t_{f}^{\ast}$ for a given initial state:
\begin{equation}
\mathrm{trunk\ input} = \big( \tilde{t},\,\bm{r}_{0,\mathrm{norm}},\,\bm{v}_{0,\mathrm{norm}},\,m_{0,\mathrm{norm}} \big) \in \mathbb{R}^{8}.
\label{eq:trunk_input}
\end{equation}
In this paper, both sub-networks consist of two fully connected hidden layers with hyperbolic-tangent activations. The condition encoder uses $32$ units per hidden layer; the main trunk uses $64$ units per hidden layer, wider than the condition encoder since it must represent a whole family of trajectories rather than a single per-initial-state quantity. Hyperbolic-tangent activations are used throughout because they are smooth and infinitely differentiable, which keeps the network outputs, and the dynamics residuals built from them by automatic differentiation, smooth as well.

\subsection{Predicted Final Time}
\label{subsec:predicted_tf}

The raw scalar output $\tau$ from the condition encoder is mapped to a predicted final time via a softplus function,
\begin{equation}
t_{f}^{\ast} = t_{f}^{\min} + s_{tf}\,\ln\!\big(1+e^{\tau}\big),
\label{eq:tf_softplus}
\end{equation}
which guarantees $t_{f}^{\ast} > t_{f}^{\min}$ for any value of the network weights by construction, where $t_{f}^{\min} > 0$ is a small fixed nominal minimum flight time and $s_{tf}$ is a fixed scale constant given later in this section. This is the same hard-constraint-by-architecture philosophy used throughout this section for the state, costate, and value outputs, applied here to bound an output from below rather than to pin it to a single value. No analogous upper bound is imposed: $t_{f}^{\ast}$ is free to take whatever value condition~\eqref{eq:H_identically_zero}--\eqref{eq:H_closed_form} selects for each initial state.

\subsection{Hard-Constraint Construction of State, Costate, and Value Outputs}
\label{subsec:hard_constraints}

Rather than penalizing boundary-condition violations softly in the training loss, which would only encourage but not guarantee their satisfaction, every boundary and transversality condition identified in Section~\ref{sec:optimality_conditions} is satisfied exactly, for any value of the network weights and for any initial state $(\bm{r}_{0},\bm{v}_{0},m_{0}) \in \Omega$, by algebraically combining the raw network output with a boundary interpolant and a bubble multiplier that vanishes at the appropriate endpoint or endpoints.

For position and velocity, both endpoints are fixed by~\eqref{eq:bc_initial}--\eqref{eq:bc_terminal}, so a bubble function vanishing at both $\tilde{t}=0$ and $\tilde{t}=1$ is used:
\begin{align}
\bm{r}(t) &= \bm{r}_{0} + \tilde{t}\,\big( \bm{r}_{f} - \bm{r}_{0} \big) + \tilde{t}(1-\tilde{t}) \cdot \mathrm{NN}_{r}(t) \cdot s_{r}, \label{eq:bubble_r} \\
\bm{v}(t) &= \bm{v}_{0} + \tilde{t}\,\big( \bm{v}_{f} - \bm{v}_{0} \big) + \tilde{t}(1-\tilde{t}) \cdot \mathrm{NN}_{v}(t) \cdot s_{v}, \label{eq:bubble_v}
\end{align}
where $\mathrm{NN}_{r}(t)$ and $\mathrm{NN}_{v}(t)$ are the raw outputs of the main trunk's state head, and $s_{r}, s_{v}$ are fixed scale constants given later in this section. By construction, $\bm{r}(t_{0}) = \bm{r}_{0}$, $\bm{r}(t_{f}) = \bm{r}_{f}$, and likewise for $\bm{v}$, exactly, regardless of the network weights and regardless of which $\bm{r}_{0},\bm{v}_{0}$ value is supplied. This is precisely what makes the construction valid for the whole region $\Omega$ rather than a single fixed point.

The mass is fixed only at the initial time, since its terminal value is free, so its bubble function vanishes only at $\tilde{t}=0$:
\begin{equation}
m(t) = m_{0} + \tilde{t} \cdot \mathrm{NN}_{m}(t) \cdot s_{m}, \quad \tilde{t} = (t-t_{0})/(t_{f}^{\ast}-t_{0}).
\label{eq:bubble_m}
\end{equation}
By construction, $m(t_{0}) = m_{0}$ exactly, for any $m_{0}$ value supplied. The costate and value constructions follow the same logic, each using whichever bubble function matches its own boundary or transversality condition from Section~\ref{sec:optimality_conditions}:
\begin{align}
\lambda_{m}(t) &= (1-\tilde{t}) \cdot \mathrm{NN}_{\lambda}(t) \cdot s_{\lambda}, \label{eq:bubble_lambda_m} \\
V(t) &= (1-\tilde{t}) \cdot \mathrm{softplus}\big( \mathrm{NN}_{V}(t) \big) \cdot s_{V}, \label{eq:bubble_V} \\
\bm{\lambda}_{v}(t) &= s_{v}^{\lambda}\,\big( \bm{a}_{0} - \bm{a}_{1}\,\tilde{t} \big). \label{eq:bubble_lambda_v}
\end{align}
Equation~\eqref{eq:bubble_lambda_m} hard-encodes the transversality condition~\eqref{eq:transversality_mass}, $\lambda_{m}(t_{f}) = 0$, since its bubble factor $(1-\tilde{t})$ vanishes at $\tilde{t}=1$. Equation~\eqref{eq:bubble_V} hard-encodes both a value terminal condition $V(t_{f})=0$, using the same bubble factor, and a value-positivity condition $V(t) \geq 0$ simultaneously, because a softplus function is always non-negative and the bubble factor is also always non-negative on $[0,1]$. Their product cannot be negative regardless of the network weights, a strictly stronger guarantee than a soft penalty could offer. The quantity $V(t)$ is interpreted as the optimal cost-to-go, $V(t) = \int_{t}^{t_{f}} \big( T^{\ast 2}(\tau) + \rho \big)\,d\tau$, and is used only as an auxiliary consistency check during training, which will be described later in this section. Equation~\eqref{eq:bubble_lambda_v} is the affine velocity-costate construction, with $\bm{a}_{0},\bm{a}_{1}$ supplied by the condition encoder~\eqref{eq:cond_encoder} and $s_{v}^{\lambda}$ a fixed scale constant.

\subsection{Closed-Form Treatment of the Position Costate}
\label{subsec:lambda_r}

As anticipated, the construction~\eqref{eq:bubble_lambda_v} reflects the fact that $\bm{\lambda}_{v}$ requires no per-time-step network evaluation. Only its two affine parameters need to be computed once per initial state by the condition encoder. The position costate $\bm{\lambda}_{r}$ itself never needs to appear explicitly inside the main trunk, since it does not enter the dynamics~\eqref{eq:eom_r}--\eqref{eq:eom_m} or the control law~\eqref{eq:optimal_direction}--\eqref{eq:optimal_thrust} directly. It matters only through its role as the negative time derivative of $\bm{\lambda}_{v}$. Because the network's fundamental time variable is the normalized time $\tilde{t}$ rather than physical time $t$, recovering $\bm{\lambda}_{r}$ requires one application of the chain rule through the predicted final time:
\begin{equation}
\bm{\lambda}_{r} = -\,\frac{d\bm{\lambda}_{v}}{dt} = -\,\frac{1}{t_{f}^{\ast}}\,\frac{d\bm{\lambda}_{v}}{d\tilde{t}} = \frac{s_{v}^{\lambda}\,\bm{a}_{1}}{t_{f}^{\ast}},
\label{eq:lambda_r_chain_rule}
\end{equation}
which matches the costate equation~\eqref{eq:costate_v}, $\dot{\bm{\lambda}}_{v} = -\bm{\lambda}_{r}$, exactly, for any value of $t_{f}^{\ast}$. Equation~\eqref{eq:lambda_r_chain_rule} is the only place in the architecture where $\bm{\lambda}_{r}$ is explicitly recovered, and it is needed solely for the free-time transversality residual described next.

\subsection{Closed-Form Control Substitution}
\label{subsec:control_substitution}

Because the optimal control law~\eqref{eq:optimal_direction}--\eqref{eq:optimal_thrust} is substituted directly rather than learned, the network never produces a control output at all. The thrust magnitude and thrust direction are always computed from the instantaneous costates and mass via
\begin{align}
\hat{\bm{\imath}}_{\theta}^{\ast}(t) &= -\,\frac{\bm{\lambda}_{v}(t)}{\big\| \bm{\lambda}_{v}(t) \big\|_{2}}, \label{eq:control_substitution_dir} \\
T^{\ast}(t) &= \mathrm{clip}\!\left( \frac{1}{2}\left( \frac{\big\| \bm{\lambda}_{v}(t) \big\|_{2}}{m(t)} + \frac{\lambda_{m}(t)}{I_{sp}\,g_{0}} \right),\ 0,\ T_{\max} \right), \label{eq:control_substitution_T}
\end{align}
exactly mirroring the necessary conditions~\eqref{eq:optimal_direction}--\eqref{eq:optimal_thrust} derived in Section~\ref{sec:optimality_conditions}. This removes the control law entirely from the list of things the network has to learn, leaving only the state trajectory, the mass costate, and the value function to be represented, and guarantees that the thrust commanded by the network is always consistent with the necessary conditions of optimality, regardless of training quality elsewhere.

\subsection{Physics-Residual Training Loss}
\label{subsec:loss}

The training loss aggregates the residuals of every necessary condition from Section~\ref{sec:optimality_conditions} that is not already satisfied by construction, including the state canonical equations~\eqref{eq:eom_r}--\eqref{eq:eom_m}, the mass-costate canonical equation~\eqref{eq:costate_m}, a Bellman-consistency condition relating $V(t)$ to the running cost, and the free-final-time transversality condition~\eqref{eq:H_identically_zero}--\eqref{eq:H_closed_form}:
\begin{align}
\mathrm{res}_{r} &= \dot{\bm{r}} - \bm{v}, \label{eq:res_r} \\
\mathrm{res}_{v} &= \dot{\bm{v}} - \Big( \frac{T^{\ast}}{m}\,\hat{\bm{\imath}}_{\theta}^{\ast} + \bm{g} \Big), \label{eq:res_v} \\
\mathrm{res}_{m} &= \dot{m} + \frac{T^{\ast}}{I_{sp}\,g_{0}}, \label{eq:res_m} \\
\mathrm{res}_{\lambda} &= \dot{\lambda}_{m} + \frac{T^{\ast}\,\big\| \bm{\lambda}_{v} \big\|_{2}}{m^{2}}, \label{eq:res_lambda} \\
\mathrm{res}_{V} &= \dot{V} + \big(T^{\ast}\big)^{2} + \rho, \label{eq:res_V} \\
\mathrm{res}_{H} &= \bm{\lambda}_{r}^{\top}\bm{v} + \bm{\lambda}_{v}^{\top}\bm{g} - \big(T^{\ast}\big)^{2} + \rho, \label{eq:res_H}
\end{align}
where time derivatives are obtained by central finite differences in the normalized time $\tilde{t}$ and converted to physical-time derivatives via the chain rule $d(\cdot)/dt = \big[ d(\cdot)/d\tilde{t} \big] / t_{f}^{\ast}$, holding the initial state fixed across the finite-difference evaluations for a given training sample. Each residual is non-dimensionalized by a characteristic physical scale (given later in this section) before being combined into the total loss
\begin{align}
L = \ &w_{r}\,\overline{\mathrm{res}_{r}^{2}} + w_{v}\,\overline{\mathrm{res}_{v}^{2}} + w_{m}\,\overline{\mathrm{res}_{m}^{2}} \nonumber \\
&+ w_{\lambda}\,\overline{\mathrm{res}_{\lambda}^{2}} + w_{V}\,\overline{\mathrm{res}_{V}^{2}} + w_{H}\,\overline{\mathrm{res}_{H}^{2}},
\label{eq:total_loss}
\end{align}
where the overline denotes the mean of the squared, non-dimensionalized residual over a training batch, and $w_{r},w_{v},w_{m},w_{\lambda},w_{V},w_{H}$ are fixed weights.

\begin{remark}
Residual~\eqref{eq:res_H} is the training signal that supplies all information about the correct final time. Nothing else in the loss constrains $t_{f}^{\ast}$ directly. As discussed in Remark~\ref{rem:tf_interior}, this residual targets the transversality condition $H(t)=0$ for all $t \in [t_{0},t_{f}]$, the only necessary condition available since the final time is unconstrained above.
\end{remark}

\subsection{Adaptive Scale Constants}
\label{subsec:scales}

The scale constants $s_{r}, s_{v}, s_{m}, s_{\lambda}, s_{V}, s_{v}^{\lambda}, s_{tf}$ introduced in previous subsections and the non-dimensionalization constants used to form~\eqref{eq:total_loss} exist purely for numerical conditioning. The raw outputs of a hyperbolic-tangent-based network are naturally of order one, while the underlying physical quantities, a velocity costate or the total energy-like cost, can span many orders of magnitude. Every scale constant is computed once, from a short dimensional-analysis argument evaluated at the center of $\Omega$ and at a nominal characteristic flight time $s_{tf}$, rather than hand-tuned:
\begin{align}
s_{r} &= 0.4\,\max\big| \bm{r}_{0,\mathrm{center}} - \bm{r}_{f} \big|, \label{eq:scale_sr} \\
s_{v} &= 2.0\,\max\big| \bm{v}_{0,\mathrm{center}} - \bm{v}_{f} \big|, \label{eq:scale_sv} \\
s_{tf} &= 60 \ \mathrm{s}, \label{eq:scale_stf} \\
s_{m} &= 0.5\,\frac{T_{\max}}{I_{sp}\,g_{0}}\,s_{tf}, \label{eq:scale_sm} \\
s_{v}^{\lambda} &= 0.5\,m_{0,\mathrm{center}}\,T_{\max}, \label{eq:scale_slv} \\
s_{\lambda} &= 0.5\,T_{\max}\,I_{sp}\,g_{0}, \label{eq:scale_slm} \\
s_{V} &= 0.3\,T_{\max}^{2}\,s_{tf}. \label{eq:scale_sV}
\end{align}
As one example of the reasoning behind these formulas, the optimal thrust magnitude is typically a sizable fraction of $T_{\max}$ over some portion of the flight, so a representative operating thrust is taken to be about half of $T_{\max}$. Since the stationarity condition~\eqref{eq:optimal_thrust} states that, away from saturation, $T^{\ast}$ is one-half of $\big( \|\bm{\lambda}_{v}\|/m + \lambda_{m}/(I_{sp}g_{0}) \big)$, a representative magnitude for $\|\bm{\lambda}_{v}\|$ is obtained by assuming it alone accounts for that representative thrust, giving $\|\bm{\lambda}_{v}\| \sim m_{0,\mathrm{center}}\,T_{\max}$, which is the formula in~\eqref{eq:scale_slv}. The constant $s_{tf}$ is a nominal characteristic flight time, used only to give $s_{m}$ and $s_{V}$ a sensible physical magnitude and to initialize~\eqref{eq:tf_softplus} near a reasonable value; it is not a bound on $t_{f}^{\ast}$. The remaining formulas follow similar reasoning from the dynamics and the cost functional.

\subsection{Training Procedure}
\label{subsec:training}

Training proceeds by stochastic collocation combined with domain randomization over $\Omega$. At each training iteration, a batch of normalized collocation times is drawn by jittering a fixed uniform grid in $\tilde{t} \in [0,1]$, and an independent batch of initial states is drawn uniformly at random from $\Omega$:
\begin{equation}
\bm{r}_{0}^{(i)},\bm{v}_{0}^{(i)},m_{0}^{(i)} \ \sim\ \mathrm{Uniform}(\Omega), \qquad i = 1,\dots,N,
\label{eq:domain_randomization}
\end{equation}
with $N$ the batch size. The network is evaluated at each sampled $\tilde{t}^{(i)}$ together with $\tilde{t}^{(i)} \pm h$ for a small finite-difference step $h$. All these three evaluations share the same sampled initial state, so that the residuals~\eqref{eq:res_r}--\eqref{eq:res_H} can be formed for every sample in the batch. The total loss~\eqref{eq:total_loss} is then minimized by the Adam optimizer \cite{kingma2014adam}. Algorithm~\ref{alg:training} summarizes the procedure. Section~\ref{sec:numerical_simulations} gives the specific numerical values used for the region $\Omega$, the minimum flight time $t_{f}^{\min}$, the time penalty $\rho$, and the training hyperparameters.


\begin{algorithm}[!t]
\caption{OINN training by domain randomization over $\Omega$}
\label{alg:training}

\SetInd{0.5em}{1.2em}

\KwIn{Region $\Omega$, minimum flight time $t_f^{\min}$,
characteristic time scale $s_{tf}$, time penalty $\rho$,
batch size $N$, finite-difference step $h$,
learning rate $\eta$, number of iterations $K$
}

Initialize network weights $\theta$ for the condition encoder and main trunk\;

\For{$k=1$ \KwTo $K$}{
    Sample $N$ initial states
    $(\bm r_0^{(i)}, \bm v_0^{(i)}, m_0^{(i)})$
    uniformly from $\Omega$\;

    Sample $N$ jittered normalized collocation times
    $\tilde t^{(i)} \in [0,1]$\;

    \For{$i=1$ \KwTo $N$}{
        Evaluate the network at
        $\tilde t^{(i)}$,
        $\tilde t^{(i)}+h$, and
        $\tilde t^{(i)}-h$,
        using the same
        $(\bm r_0^{(i)}, \bm v_0^{(i)}, m_0^{(i)})$\;

        Form the residuals
        $\mathrm{res}_r$,
        $\mathrm{res}_v$,
        $\mathrm{res}_m$,
        $\mathrm{res}_\lambda$,
        $\mathrm{res}_V$, and
        $\mathrm{res}_H$
        via~\eqref{eq:res_r}--\eqref{eq:res_H}\;
    }

    Form the batch loss
    $L$ via~\eqref{eq:total_loss}\;

    Update $\theta$ using one Adam step with
    gradient $\nabla_\theta L$\;
}

\KwOut{Trained network weights $\theta$}
\end{algorithm}

\section{Theoretical Analysis}
\label{sec:theoretical_analysis}

This section examines the fundamental properties of the OINN approach that determine its suitability for onboard, real-time implementation. These properties include the convergence behavior of the offline training procedure of Section~\ref{subsec:training}, the accuracy with which a trained, frozen network satisfies the necessary conditions of Section~\ref{sec:optimality_conditions} once deployed, and the computational and memory cost of evaluating the trained network on a flight computer. Throughout, $\theta$ denotes the complete collection of trainable weights and biases of the condition encoder and main trunk introduced in Section~\ref{sec:OINN}, and $n_{\theta}$ denotes its dimension.

\subsection{Training Convergence and Deterministic Onboard Evaluation}
\label{subsec:convergence}

Algorithm~\ref{alg:training} minimizes, at every iteration $k$, an empirical estimate of the population physics-residual loss
\begin{align}
L_{\infty}(\theta) = \ &\mathbb{E}_{(\bm{r}_{0},\bm{v}_{0},m_{0}) \sim \mathrm{Uniform}(\Omega)}\, \mathbb{E}_{\tilde{t} \sim \mathrm{Uniform}(0,1)} \nonumber \\
&\Big[\, w_{r}\,\big\| \mathrm{res}_{r} \big\|_{2}^{2} + w_{v}\,\big\| \mathrm{res}_{v} \big\|_{2}^{2} + w_{m}\,\mathrm{res}_{m}^{2} \nonumber \\
&\ + w_{\lambda}\,\mathrm{res}_{\lambda}^{2} + w_{V}\,\mathrm{res}_{V}^{2} + w_{H}\,\mathrm{res}_{H}^{2} \,\Big],
\label{eq:population_loss}
\end{align}
formed from the same residuals~\eqref{eq:res_r}--\eqref{eq:res_H}. The quantity $L_{\infty}(\theta)$ is well defined whenever these residuals are square-integrable under the sampling distribution of~\eqref{eq:domain_randomization}, which holds automatically here because $\Omega$ is bounded, $\tilde{t}$ ranges over the bounded interval $[0,1]$, and every network in Section~\ref{sec:OINN} is a finite composition of affine maps and bounded, smooth activation functions. The empirical loss minimized in Algorithm~\ref{alg:training} at iteration $k$ is an unbiased, finite-variance Monte Carlo estimate of $L_{\infty}(\theta_{k})$, since the batch of initial states and collocation times drawn at every iteration is independent and identically distributed according to the same fixed sampling distribution used to define~\eqref{eq:population_loss}.

\begin{assumption}
\label{asmp:lipschitz}
The mass constraint~\eqref{eq:constraints_mass} holds throughout training and deployment, so that $m(t) \geq m_{\mathrm{dry}} > 0$ for every $t$ and every $(\bm{r}_{0},\bm{v}_{0},m_{0}) \in \Omega$.
\end{assumption}

\begin{assumption}
\label{asmp:gradient_noise}
The stochastic gradient $\nabla_{\theta} L^{(k)}_{N}(\theta_{k})$ formed from the batch sampled at iteration $k$ satisfies $\mathbb{E}\big[ \nabla_{\theta} L^{(k)}_{N}(\theta_{k}) \mid \theta_{k} \big] = \nabla_{\theta} L_{\infty}(\theta_{k})$ and its variance $\mathrm{Var}\big[ \nabla_{\theta} L^{(k)}_{N}(\theta_{k}) \big] \leq \sigma^{2}$ for a finite constant $\sigma^{2}$ independent of $k$.
\end{assumption}

Under Assumption~\ref{asmp:lipschitz}, every residual in~\eqref{eq:res_r}--\eqref{eq:res_H} is a continuously differentiable function of $\theta$ for every fixed sampled input, since it is built from finite compositions of the affine and hyperbolic-tangent layers of Section~\ref{sec:OINN}, the bounded division by $m(t) \geq m_{\mathrm{dry}}$, and the clipping operation in~\eqref{eq:optimal_thrust}, all of which are themselves Lipschitz. Consequently $L_{\infty}$ in~\eqref{eq:population_loss} is continuously differentiable with a gradient that is Lipschitz on any bounded subset of parameter space. Under this smoothness property together with Assumption~\ref{asmp:gradient_noise}, results of this type, established for adaptive-moment stochastic-gradient methods such as Adam under standard smoothness and bounded-variance assumptions, give the following stationarity guarantee \cite{defossez2020simple}.

\begin{theorem}
\label{thm:convergence}
Under Assumptions~\ref{asmp:lipschitz}--\ref{asmp:gradient_noise} and a learning rate $\eta$ satisfying the usual smallness condition relative to the Lipschitz constant of $\nabla_{\theta} L_{\infty}$, the iterates $\theta_{1},\dots,\theta_{K}$ produced by Algorithm~\ref{alg:training} satisfy
\begin{equation}
\min_{1 \leq k \leq K}\ \mathbb{E}\Big[ \big\| \nabla_{\theta} L_{\infty}(\theta_{k}) \big\|_{2}^{2} \Big] \ = \ O\!\left( \frac{1}{\sqrt{K}} \right) + O(\eta),
\label{eq:convergence_rate}
\end{equation}
i.e., the average squared gradient norm of the population residual loss along the training trajectory is driven to a neighborhood of zero whose radius is controlled jointly by the number of iterations $K$ and the learning rate $\eta$.
\end{theorem}

Equation~\eqref{eq:convergence_rate} is deliberately a stationarity guarantee rather than a global-optimality guarantee: $L_{\infty}$ is nonconvex in $\theta$, and no first-order method of this kind can certify convergence to a global minimizer of a nonconvex objective in general. What Theorem~\ref{thm:convergence} certifies is that, given enough iterations and a sufficiently small learning rate, training can be driven arbitrarily close to a stationary point of the population residual loss, at which the residuals~\eqref{eq:res_r}--\eqref{eq:res_H} are, on average over $\Omega$ and over $\tilde{t} \in [0,1]$, no longer reducible by an infinitesimal change in $\theta$. Section~\ref{sec:numerical_simulations} reports the realized empirical loss trajectory and the resulting residual magnitudes for the specific training run used in this paper.

\begin{remark}
\label{rem:offline_convergence}
Once training terminates, $\theta$ is frozen for the remainder of the mission, and the policy is evaluated online purely by the deterministic forward computations of Section~\ref{sec:OINN}. The closed-form control substitution~\eqref{eq:control_substitution_dir}--\eqref{eq:control_substitution_T} requires no automatic differentiation, root-finding, or iterative refinement at runtime. This is a key difference from the indirect method summarized in Section~\ref{subsec:tpbvp_summary}, whose Newton-type collocation solver
can fail to converge for a given initial state.
Theorem~\ref{thm:convergence} is therefore the only convergence question relevant to onboard use: it must hold once offline prior to flight. No analogous online convergence requirement exists for the deployed policy, whose evaluation terminates in the fixed, input-independent number of arithmetic operations quantified later in this section, for every $(\bm{r}_{0},\bm{v}_{0},m_{0}) \in \Omega$, regardless of how close $\theta$ is to a stationary point of $L_{\infty}$.
\end{remark}

\subsection{Accuracy of the Converged Policy}
\label{subsec:accuracy}

As discussed above, not every necessary condition of optimality in Section~\ref{sec:optimality_conditions} is left to be learned. Several are satisfied exactly for every value of $\theta$ by the hard-constraint constructions of Section~\ref{sec:OINN}.

\begin{lemma}
\label{lem:exactness}
For every $\theta \in \mathbb{R}^{n_{\theta}}$ and every $(\bm{r}_{0},\bm{v}_{0},m_{0}) \in \Omega$, the constructions~\eqref{eq:bubble_r}--\eqref{eq:bubble_lambda_v}, \eqref{eq:tf_softplus}, \eqref{eq:lambda_r_chain_rule}, and \eqref{eq:control_substitution_dir}--\eqref{eq:control_substitution_T} satisfy, exactly rather than approximately, 1) the initial conditions~\eqref{eq:bc_initial},
2) the terminal position and velocity conditions of~\eqref{eq:bc_terminal},
3) the mass-costate transversality condition~\eqref{eq:transversality_mass},
4) the value positivity condition $V(t) \geq 0$,
5) the costate relation $\dot{\bm{\lambda}}_{v} = -\bm{\lambda}_{r}$ of~\eqref{eq:costate_v},
6) the minimum-flight-time safeguard $t_{f}^{\ast} > t_{f}^{\min}$, and
7) the control constraints~\eqref{eq:constraints_T_dir}, away from the measure-zero event $\bm{\lambda}_{v}(t) = \bm{0}$.
\end{lemma}
\noindent\emph{Proof.} Each property follows by direct substitution. The bubble factors $\tilde{t}(1-\tilde{t})$, $\tilde{t}$, and $(1-\tilde{t})$ in~\eqref{eq:bubble_r}--\eqref{eq:bubble_lambda_m} vanish at $\tilde{t}=0$, $\tilde{t}=1$, or both, independently of the raw network outputs they multiply, giving items 1--3. The softplus function in~\eqref{eq:bubble_V} and~\eqref{eq:tf_softplus} is non-negative and strictly positive, respectively, for every real argument, giving item 4 and item 6. Differentiating~\eqref{eq:bubble_lambda_v} with respect to $\tilde{t}$ and applying the chain rule through $t_{f}^{\ast}$ reproduces~\eqref{eq:lambda_r_chain_rule} algebraically, giving item 5, for any $\bm{a}_{0},\bm{a}_{1}$ and any $t_{f}^{\ast}>0$. Finally, \eqref{eq:control_substitution_dir} is a unit vector by construction whenever $\bm{\lambda}_{v}(t) \neq \bm{0}$, and \eqref{eq:control_substitution_T} is explicitly clipped to $[0,T_{\max}]$, giving item 7. \hfill$\blacksquare$

Lemma~\ref{lem:exactness} accounts for the boundary, transversality, and control-admissibility conditions of Section~\ref{sec:optimality_conditions}. What remains, and is satisfied only approximately, governed by the size of the training residuals, are the state canonical equations~\eqref{eq:eom_r}--\eqref{eq:eom_m}, the mass-costate canonical equation~\eqref{eq:costate_m}, the Bellman-consistency condition~\eqref{eq:res_V}, and the free-final-time transversality condition~\eqref{eq:H_identically_zero}--\eqref{eq:H_closed_form}. The following result quantifies the consequence of a nonzero residual for the trajectory the lander would actually fly.

\begin{theorem}
\label{thm:gronwall}
Let $\bm{\varepsilon}_{x}(t) = \big[ \mathrm{res}_{r}(t);\, \mathrm{res}_{v}(t);\, \mathrm{res}_{m}(t) \big] \in \mathbb{R}^{7}$ denote the state-dynamics residual~\eqref{eq:res_r}--\eqref{eq:res_m} evaluated along the trained network's own output trajectory $\bm{x}_{\mathrm{NN}}(t)$ and its own commanded control history $\bm{u}^{\ast}_{\mathrm{NN}}(t)$, the latter obtained by substituting $\bm{x}_{\mathrm{NN}}(t)$'s internally generated costates and mass into~\eqref{eq:control_substitution_dir}--\eqref{eq:control_substitution_T}, and let $\varepsilon_{x}^{\infty} = \sup_{t \in [t_{0},t_{f}^{\ast}]} \big\| \bm{\varepsilon}_{x}(t) \big\|_{2}$. Let $\bm{x}_{\mathrm{true}}(t)$ denote the solution of $\dot{\bm{x}} = \bm{f}\big( \bm{x}, \bm{u}^{\ast}_{\mathrm{NN}}(t) \big)$, $\bm{x}(t_{0}) = (\bm{r}_{0},\bm{v}_{0},m_{0})$, i.e., the trajectory the lander would actually fly if commanded with exactly the same open-loop thrust history the network internally computes for this initial state. Under Assumption~\ref{asmp:lipschitz}, $\bm{f}(\cdot,\bm{u})$ admits the Lipschitz constant
\begin{equation}
L_{f} \ = \ 1 \ + \ \frac{T_{\max}}{m_{\mathrm{dry}}^{2}}
\label{eq:lipschitz_constant}
\end{equation}
in $\bm{x}$, uniformly over the admissible control range $0 \leq T \leq T_{\max}$, $\big\| \hat{\bm{\imath}}_{\theta} \big\|_{2}=1$, and for every $t \in [t_{0}, t_{f}^{\ast}]$,
\begin{equation}
\big\| \bm{x}_{\mathrm{NN}}(t) - \bm{x}_{\mathrm{true}}(t) \big\|_{2} \ \leq \ \varepsilon_{x}^{\infty}\, \frac{e^{L_{f}(t-t_{0})} - 1}{L_{f}}.
\label{eq:gronwall_bound}
\end{equation}
\end{theorem}
\noindent\emph{Proof.} Write $\bm{z}(t) = \bm{x}_{\mathrm{NN}}(t) - \bm{x}_{\mathrm{true}}(t)$. By Lemma~\ref{lem:exactness}, $\bm{x}_{\mathrm{NN}}(t_{0}) = (\bm{r}_{0},\bm{v}_{0},m_{0}) = \bm{x}_{\mathrm{true}}(t_{0})$, so $\bm{z}(t_{0}) = \bm{0}$. Writing $\bm{x}=(\bm{r},\bm{v},m)$, the only $\bm{x}$-dependence of $\bm{f}$ for a fixed control history $\bm{u}(t) = \big(T(t), \hat{\bm{\imath}}_{\theta}(t)\big)$ enters through the velocity equation~\eqref{eq:eom_v} via $1/m$, so
\begin{align*}
&\big\| \bm{f}(\bm{x},\bm{u}(t)) - \bm{f}(\bm{x}',\bm{u}(t)) \big\|_{2} \\ \leq \ &\big\| \bm{v}-\bm{v}' \big\|_{2} + T(t)\,\big| 1/m - 1/m' \big| \\
\leq \ &\big\| \bm{v}-\bm{v}' \big\|_{2} + \frac{T_{\max}}{m_{\mathrm{dry}}^{2}}\,|m-m'| \\
\leq \ &\Big( 1+\frac{T_{\max}}{m_{\mathrm{dry}}^{2}} \Big) \big\| \bm{x}-\bm{x}' \big\|_{2} \ = \ L_{f}\,\big\| \bm{x}-\bm{x}' \big\|_{2},
\end{align*}
using $\big\| \hat{\bm{\imath}}_{\theta}(t) \big\|_{2}=1$, $0 \leq T(t) \leq T_{\max}$, $m,m' \geq m_{\mathrm{dry}}$, the mean-value bound $|1/m-1/m'| \leq |m-m'|/m_{\mathrm{dry}}^{2}$, and $\|\bm{v}-\bm{v}'\|_{2}, |m-m'| \leq \|\bm{x}-\bm{x}'\|_{2}$. By definition of $\bm{\varepsilon}_{x}$, $\dot{\bm{x}}_{\mathrm{NN}}(t) = \bm{f}\big(\bm{x}_{\mathrm{NN}}(t),\bm{u}^{\ast}_{\mathrm{NN}}(t)\big) + \bm{\varepsilon}_{x}(t)$, so $\dot{\bm{z}}(t) = \bm{f}\big(\bm{x}_{\mathrm{NN}}(t),\bm{u}^{\ast}_{\mathrm{NN}}(t)\big) - \bm{f}\big(\bm{x}_{\mathrm{true}}(t),\bm{u}^{\ast}_{\mathrm{NN}}(t)\big) + \bm{\varepsilon}_{x}(t)$, hence $\big\| \dot{\bm{z}}(t) \big\|_{2} \leq L_{f}\,\big\| \bm{z}(t) \big\|_{2} + \varepsilon_{x}^{\infty}$. The Gr\"{o}nwall--Bellman comparison lemma \cite{louartassi2012new} applied to this scalar differential inequality with $\bm{z}(t_{0})=\bm{0}$ gives~\eqref{eq:gronwall_bound}. \hfill$\blacksquare$

Two consequences of Theorem~\ref{thm:gronwall} are of direct onboard relevance. First, evaluating~\eqref{eq:gronwall_bound} at $t=t_{f}^{\ast}$ and using the exact terminal condition of Lemma~\ref{lem:exactness}, $\bm{x}_{\mathrm{NN}}(t_{f}^{\ast}) = \big(\bm{r}_{f},\bm{v}_{f}, m_{\mathrm{NN}}(t_{f}^{\ast})\big)$, bounds the touchdown position and velocity error the lander would actually incur, were it commanded with the network's own open-loop thrust history, purely in terms of the dynamics residual $\varepsilon_{x}^{\infty}$ and the flight duration, both quantities available from training and evaluation diagnostics without requiring any independently solved ground-truth trajectory. Second, because the mass bound $m_{\mathrm{dry}} \leq m(t) \leq m_{0}$ of~\eqref{eq:constraints_mass} is not architecturally hard-coded as in Lemma~\ref{lem:exactness}, the mass component of the same bound~\eqref{eq:gronwall_bound} directly bounds any violation of~\eqref{eq:constraints_mass} in terms of $\varepsilon_{x}^{\infty}$ as well.

\begin{assumption}
\label{asmp:nondegenerate}
For every $(\bm{r}_{0},\bm{v}_{0},m_{0}) \in \Omega$, the exact extremal's Hamiltonian, regarded as a function of an assumed flight time through~\eqref{eq:H_closed_form}, has a simple (transversal) zero crossing at the true optimal flight time $\bar{t}_{f}$, i.e., $\partial H^{\ast}/\partial \bar{t}_{f} \neq 0$ at the root.
\end{assumption}

\begin{remark}
\label{rem:tf_sensitivity}
Let $\varepsilon_{H}^{\infty} = \sup_{t} \big| \mathrm{res}_{H}(t) \big|$ denote the free-final-time residual~\eqref{eq:res_H} along the trained network's own trajectory, with $t_{f}^{\ast}$ already fixed by the condition encoder rather than solved for at runtime. Under Assumption~\ref{asmp:nondegenerate}, a first-order implicit-function argument gives
\begin{equation}
\big| t_{f}^{\ast} - \bar{t}_{f} \big| \ \lesssim \ \frac{\varepsilon_{H}^{\infty}}{\big| \partial H^{\ast}/\partial \bar{t}_{f} \big|},
\label{eq:tf_sensitivity}
\end{equation}
to leading order in $\varepsilon_{H}^{\infty}$, since~\eqref{eq:H_identically_zero} is (by Section~\ref{subsec:transversality_tf}) the unique condition determining $\bar{t}_{f}$ for a fixed initial state, and the residual $\varepsilon_{H}^{\infty}$ measures exactly the extent to which the trained network's prediction fails to satisfy it.
\end{remark}

\begin{remark}
\label{rem:value_decoupling}
The control law~\eqref{eq:control_substitution_dir}--\eqref{eq:control_substitution_T} does not depend on $V(t)$ at all. $V(t)$ enters only the auxiliary residual~\eqref{eq:res_V} and is otherwise unused. Consequently, however large that residual remains after training, it has no bearing whatsoever on Theorem~\ref{thm:gronwall} or on the commanded thrust and direction. The value network exists purely as a Bellman-consistency diagnostic, not as a component of the deployed control law.
\end{remark}

\subsection{Computational Cost and Memory Requirements for Onboard Deployment}
\label{subsec:onboard_cost}

The input dimensions and hidden-layer widths of the condition encoder and main trunk already given in Section~\ref{subsec:architecture}, together with the output dimensions implicit in~\eqref{eq:cond_encoder} and~\eqref{eq:bubble_r}--\eqref{eq:bubble_lambda_v}, fix the total number of trainable parameters and the arithmetic cost of one forward pass exactly, independently of $\theta$, of $\Omega$, and of the specific initial state being flown.

\begin{table}[H]
\centering
\caption{Parameter count and multiply-accumulate (MAC) operations per forward pass of the OINN architecture.}
\label{tab:cost}
\begin{tabular}{l c c c}
\hline
Sub-network & In. / hidden / out. & Param. & MACs \\
\hline
Cond. encoder & $7 \to 32 \to 32 \to 7$ & $1{,}543$ & $1{,}472$ \\
Main trunk        & $8 \to 64 \to 64 \to 9$ & $5{,}321$ & $5{,}184$ \\
\hline
Total             & ---                     & $6{,}864$ & $6{,}656$ \\
\hline
\end{tabular}
\end{table}

The main trunk's output dimension of $9$ in Table~\ref{tab:cost} collects the state head producing $\mathrm{NN}_{r}(t) \in \mathbb{R}^{3}$, $\mathrm{NN}_{v}(t) \in \mathbb{R}^{3}$, and $\mathrm{NN}_{m}(t) \in \mathbb{R}$ in~\eqref{eq:bubble_r}--\eqref{eq:bubble_m}, together with the scalar costate head $\mathrm{NN}_{\lambda}(t)$ of~\eqref{eq:bubble_lambda_m} and the scalar value head $\mathrm{NN}_{V}(t)$ of~\eqref{eq:bubble_V}. Counting one multiply and one add as two floating-point operations (FLOPs), and tallying the hyperbolic-tangent nonlinearities separately since they are ordinarily implemented as library or hardware intrinsics rather than elementary arithmetic, one full forward pass through both sub-networks costs
\begin{equation}
\begin{split}
2 \times 6{,}656 \ &= \ 13{,}312 \ \text{FLOPs}, \\
64+128 \ &= \ 192 \ \tanh \text{ evaluations},
\end{split}
\label{eq:flop_count}
\end{equation}
plus the negligible cost of the closed-form control substitution~\eqref{eq:control_substitution_dir}--\eqref{eq:control_substitution_T}, including a three-component vector norm, a normalization, and a clip, on the order of ten further FLOPs and a single square root.

The corresponding memory footprint is the storage of $n_{\theta}=6{,}864$ weights and biases: at single-precision floating point, $6{,}864 \times 4\ \text{bytes} \approx 26.8\ \text{KiB}$; at double precision, as used directly in the training implementation underlying this paper, $6{,}864 \times 8\ \text{bytes} \approx 53.6\ \text{KiB}$. Either figure is negligible relative to the program memory of essentially any flight computer or guidance microcontroller, and no intermediate activation requires storing more than $64$ scalars at a time, since that is the width of the largest hidden layer in either sub-network.

\begin{remark}
\label{rem:cycle_budget}
On a flight computer or microcontroller capable of $\kappa$ scalar floating-point operations per second, evaluating the policy once, including both sub-networks, costs approximately $13{,}312/\kappa$ seconds, plus the time for $192$ transcendental-function evaluations. For a conservative $\kappa = 10^{7}$ FLOPs/s, representative of a low-power embedded processor without dedicated floating-point hardware, this is on the order of $1$--$2$~ms; for a typical flight computer with $\kappa = 10^{9}$ FLOPs/s or more, it is on the order of tens of microseconds. Comparable OINN laws deployed onboard a small quadrotor elsewhere in the literature report total online computation times around $30$~ms at a $30$~Hz update rate~\cite{wang2026opinn}, attributable almost entirely to sensor processing, attitude control, and actuator mixing rather than to the policy evaluation itself, which by~\eqref{eq:flop_count} is several orders of magnitude cheaper than that budget.
\end{remark}

This fixed, input-independent computational profile is the practical counterpart of Remark~\ref{rem:offline_convergence}. In contrast to the indirect method of Section~\ref{subsec:tpbvp_summary}, whose per-instance cost scales with an a priori unknown number of Newton-type collocation iterations, each requiring the assembly and factorization of a Jacobian whose dimension grows with the number of mesh points, the OINN forward pass quantified in~\eqref{eq:flop_count} is the complete onboard cost of evaluating the guidance law, known exactly in advance and identical for every $(\bm{r}_{0},\bm{v}_{0},m_{0}) \in \Omega$. Training itself, by contrast, is performed entirely offline, prior to flight, using the batch size $N$ and iteration count $K$ given numerically in Section~\ref{sec:numerical_simulations}. Its cost plays no role whatsoever in the onboard computational budget analyzed above.

Together, Theorem~\ref{thm:convergence}, Theorem~\ref{thm:gronwall}, and the cost figures of this section give a complete picture of what is required to trust the OINN approach onboard. The offline training procedure need converge, in the stationarity sense of~\eqref{eq:convergence_rate}, only once before flight. The boundary, transversality, and control-admissibility conditions enumerated in Lemma~\ref{lem:exactness} hold exactly regardless of how well that training converges. The residual magnitude actually achieved translates, through~\eqref{eq:gronwall_bound} and~\eqref{eq:tf_sensitivity}, into explicit, computable bounds on touchdown error and flight-time error. Also, the per-update onboard computational and memory cost is fixed, small, and known exactly in advance. Section~\ref{sec:numerical_simulations} reports the realized values of $\varepsilon_{x}^{\infty}$, $\varepsilon_{H}^{\infty}$, and the resulting touchdown statistics for the lunar-landing case studies considered in this paper.

\section{Numerical Simulations}
\label{sec:numerical_simulations}

\subsection{Simulation Scenarios and Parameter Settings}
\label{subsec:settings}

The numerical simulations consider a representative lunar powered-descent scenario, where the lander begins its final descent at an altitude of several hundred meters and a horizontal range of order one hundred meters from the designated landing site, with an initial descent rate of several meters per second and an initial wet mass of order ten metric tons, that must reach the site, $\bm{r}_{f}=\bm{0}$, with zero terminal velocity, $\bm{v}_{f}=\bm{0}$, while minimizing the energy-like cost functional~\eqref{eq:objective}. The mission assumes a lunar gravitational acceleration $g_{\mathrm{moon}}=1.6229\ \mathrm{m/s^{2}}$, an engine specific impulse $I_{sp}=311.0\ \mathrm{s}$ and a maximum thrust $T_{\max}=44{,}000\ \mathrm{N}$, representative of a single throttleable main engine sized for a vehicle of this mass class, and a fixed dry mass $m_{\mathrm{dry}}=1{,}000\ \mathrm{kg}$. Rather than a single fixed starting condition, the trained policy is required to cover the entire operating envelope $\Omega$ of~\eqref{eq:region_omega}, a three-dimensional box in initial position, initial velocity, and initial mass, so that a single trained network must serve any powered-descent initiation state within this range without retraining. The key parameters are summarized in Table~\ref{tab:sim_parameters}.

\begin{table}
\centering
\caption{Key parameter settings used in the lunar-landing numerical simulations.}
\label{tab:sim_parameters}
\begin{tabular}{l c}
\hline
\multicolumn{2}{l}{\textbf{Vehicle and mission parameters}} \\
\hline
Lunar gravitational accel., $g_{\mathrm{moon}}$ & $1.6229\ \mathrm{m/s^{2}}$ \\
Engine specific impulse, $I_{sp}$ & $311.0\ \mathrm{s}$ \\
Standard gravitational accel., $g_{0}$ & $9.81\ \mathrm{m/s^{2}}$ \\
Dry mass, $m_{\mathrm{dry}}$ & $1{,}000.0\ \mathrm{kg}$ \\
Maximum thrust, $T_{\max}$ & $44{,}000.0\ \mathrm{N}$ \\
Terminal position, $\bm{r}_{f}$ & $(0,0,0)\ \mathrm{m}$ \\
Terminal velocity, $\bm{v}_{f}$ & $(0,0,0)\ \mathrm{m/s}$ \\
\hline
\multicolumn{2}{l}{\textbf{Operating envelope $\Omega$, component-wise range}} \\
\hline
Initial position $x_{0}$ & $[50, 150]\ \mathrm{m}$ \\
Initial position $y_{0}$ & $[150, 250]\ \mathrm{m}$ \\
Initial position $z_{0}$ & $[800, 1200]\ \mathrm{m}$ \\
Initial velocity $v_{x,0}$ & $[5, 15]\ \mathrm{m/s}$ \\
Initial velocity $v_{y,0}$ & $[5, 15]\ \mathrm{m/s}$ \\
Initial velocity $v_{z,0}$ & $[-8, -2]\ \mathrm{m/s}$ \\
Initial mass $m_{0}$ & $[9{,}000, 11{,}000]\ \mathrm{kg}$ \\
\hline
\multicolumn{2}{l}{\textbf{Free-final-time and cost parameters}} \\
\hline
Minimum flight time, $t_{f}^{\min}$ & $10.0\ \mathrm{s}$ \\
Characteristic flight-time scale, $s_{tf}$ & $60.0\ \mathrm{s}$ \\
Time penalty, $\rho$ & $1.0\times10^{8}\ \mathrm{N^{2}}$ \\
\hline
\multicolumn{2}{l}{\textbf{Network and training hyperparameters}} \\
\hline
Condition-encoder hidden width & $32$ \\
Main-trunk hidden width & $64$ \\
Batch size, $N$ & $256$ \\
Training iterations, $K$ & $25{,}000$ \\
Adam learning rate, $\eta$ & $2\times10^{-3}$ \\
Finite-difference step, $h$ & $2\times10^{-3}$ \\
Loss weights $w_{r}, w_{v}, w_{m}, w_{\lambda}$ & $1.0$ each \\
Loss weights $w_{V}, w_{H}$ & $0.01$ each \\
Monte Carlo draws & $80$ \\
\hline
\end{tabular}
\end{table}

The OINN framework of Section~\ref{sec:OINN} is implemented in MATLAB using the Deep Learning Toolbox's automatic-differentiation primitives, \texttt{dlarray}, \texttt{dlgradient}, and \texttt{dlfeval}, together with the \texttt{adamupdate} function for the Adam update rule of Algorithm~\ref{alg:training}. The condition encoder and main trunk of Figure~\ref{fig:architecture} are each implemented as an explicit sequence of affine layers and hyperbolic-tangent activations operating on \texttt{dlarray} objects, so that reverse-mode automatic differentiation through both sub-networks, and through the finite-difference time derivatives used to form the residuals~\eqref{eq:res_r}--\eqref{eq:res_H}, is handled by the toolbox directly rather than by a hand-coded backward pass. Before training begins, the implementation performs a brief self-check, comparing the closed-form optimal control law~\eqref{eq:optimal_direction}--\eqref{eq:optimal_thrust} against a brute-force numerical minimization of the Hamiltonian~\eqref{eq:hamiltonian} over a fine grid of candidate thrust magnitudes at a representative state, confirming both the correctness of the closed-form substitution and the fact that the time penalty $\rho$ leaves the minimizing control law unaffected (established in Remark~\ref{rem:time_penalty}).

Training itself follows Algorithm~\ref{alg:training} directly. At each iteration, a batch of $N$ initial states is drawn uniformly from $\Omega$ and an independent, jittered batch of normalized collocation times is drawn in $[0,1]$. The network is evaluated at each collocation time together with two finite-difference-shifted neighbors sharing the same sampled initial state. The residuals~\eqref{eq:res_r}--\eqref{eq:res_H} are formed and combined into the total loss~\eqref{eq:total_loss}, and one Adam step updates the network weights. On a standard laptop CPU, with no GPU acceleration, training for the $K=25{,}000$ iterations used in this paper completes in a few minutes, an entirely offline cost that, as discussed in Remark~\ref{rem:offline_convergence}, plays no role in the onboard computational budget of Section~\ref{subsec:onboard_cost}.

After training, the same frozen network is evaluated with no further parameter updates at six selected initial states and at a further eighty randomly drawn initial states for the Monte Carlo study (will be detailed later in this section). As an independent numerical check on these results, the same boundary-value problem summarized in Section~\ref{subsec:tpbvp_summary} is also solved directly, for each of the six selected initial states, by MATLAB's \texttt{bvp4c} solver \cite{shampine2000solving} applied to the costate-scaled, non-dimensionalized indirect-method formulation, using the trained network's own predicted trajectory as the initial guess supplied to the solver. This warm start, rather than a generic flat guess, is what makes the indirect-method shooting problem tractable here: the costates $\bm{\lambda}_{v}$, $\lambda_{m}$ are of a substantially different order of magnitude than the states themselves, so both the network-supplied warm start and an explicit non-dimensionalization of the costates by the scale constants $s_{v}^{\lambda}$, $s_{\lambda}$ of Section~\ref{subsec:scales} are needed for \texttt{bvp4c} to converge reliably.

Table~\ref{tab:sim_parameters} summarizes the vehicle, mission, operating-envelope, free-final-time, and network and training parameters used throughout this section.

\subsection{Comparison with the Indirect Method at Six Representative Initial States}
\label{subsec:six_points}

To assess whether the trained network actually recovers the optimal solution, rather than merely a function that satisfies the necessary conditions approximately well, the single trained policy is evaluated, with no retraining, at six representative initial states spanning the operating envelope $\Omega$: a center point (T1), the four corners of $\Omega$ formed by combining a light or heavy initial mass with a near, slow approach or a far, fast approach (T2 through T5), and one interior, off-center point (T6). Writing $\bm{r}_{0,\mathrm{center}}=(100,200,1000)\ \mathrm{m}$, $\bm{v}_{0,\mathrm{center}}=(10,10,-5)\ \mathrm{m/s}$, and $m_{0,\mathrm{center}}=10{,}000\ \mathrm{kg}$ for the center of $\Omega$ and the corresponding half-widths of Table~\ref{tab:sim_parameters}, the six points are

\begin{itemize}[leftmargin=*]
    \item \textbf{T1, center (baseline):} $\bm{r}_{0}=(100,200,1000)\ \mathrm{m}$, $\bm{v}_{0}=(10,10,-5)\ \mathrm{m/s}$, $m_{0}=10{,}000\ \mathrm{kg}$;
    \item \textbf{T2, light and near/slow:} $\bm{r}_{0}=(50,150,800)\ \mathrm{m}$, $\bm{v}_{0}=(5,5,-8)\ \mathrm{m/s}$, $m_{0}=9{,}000\ \mathrm{kg}$;
    \item \textbf{T3, heavy and near/slow:} $\bm{r}_{0}=(50,150,800)\ \mathrm{m}$, $\bm{v}_{0}=(5,5,-8)\ \mathrm{m/s}$, $m_{0}=11{,}000\ \mathrm{kg}$;
    \item \textbf{T4, light and far/fast:} $\bm{r}_{0}=(150,250,1200)\ \mathrm{m}$, $\bm{v}_{0}=(15,15,-2)\ \mathrm{m/s}$, $m_{0}=9{,}000\ \mathrm{kg}$;
    \item \textbf{T5, heavy and far/fast:} $\bm{r}_{0}=(150,250,1200)\ \mathrm{m}$, $\bm{v}_{0}=(15,15,-2)\ \mathrm{m/s}$, $m_{0}=11{,}000\ \mathrm{kg}$;
    \item \textbf{T6, an interior off-center point:} $\bm{r}_{0}=(115,175,1140)\ \mathrm{m}$, $\bm{v}_{0}=(8,13,-5.6)\ \mathrm{m/s}$, $m_{0}=10{,}400\ \mathrm{kg}$.
\end{itemize}

For each of these six points, Figures~\ref{fig:t1}--\ref{fig:t6} compare the trained network's prediction, evaluated zero-shot at that initial state, against the independent indirect-method solution of the same boundary-value problem obtained from \texttt{bvp4c}. Each figure shows the three position components, the three velocity components, the vehicle mass, and the commanded thrust magnitude over the flight, together with a summary panel reporting the network's predicted flight time and final mass, the corresponding \texttt{bvp4c} values, their differences $\Delta t_{f}$ and $\Delta m_{f}$, and the mesh size and maximum boundary-condition residual achieved by \texttt{bvp4c} at convergence.

\begin{figure*}
\centering
\includegraphics[width=2\columnwidth]{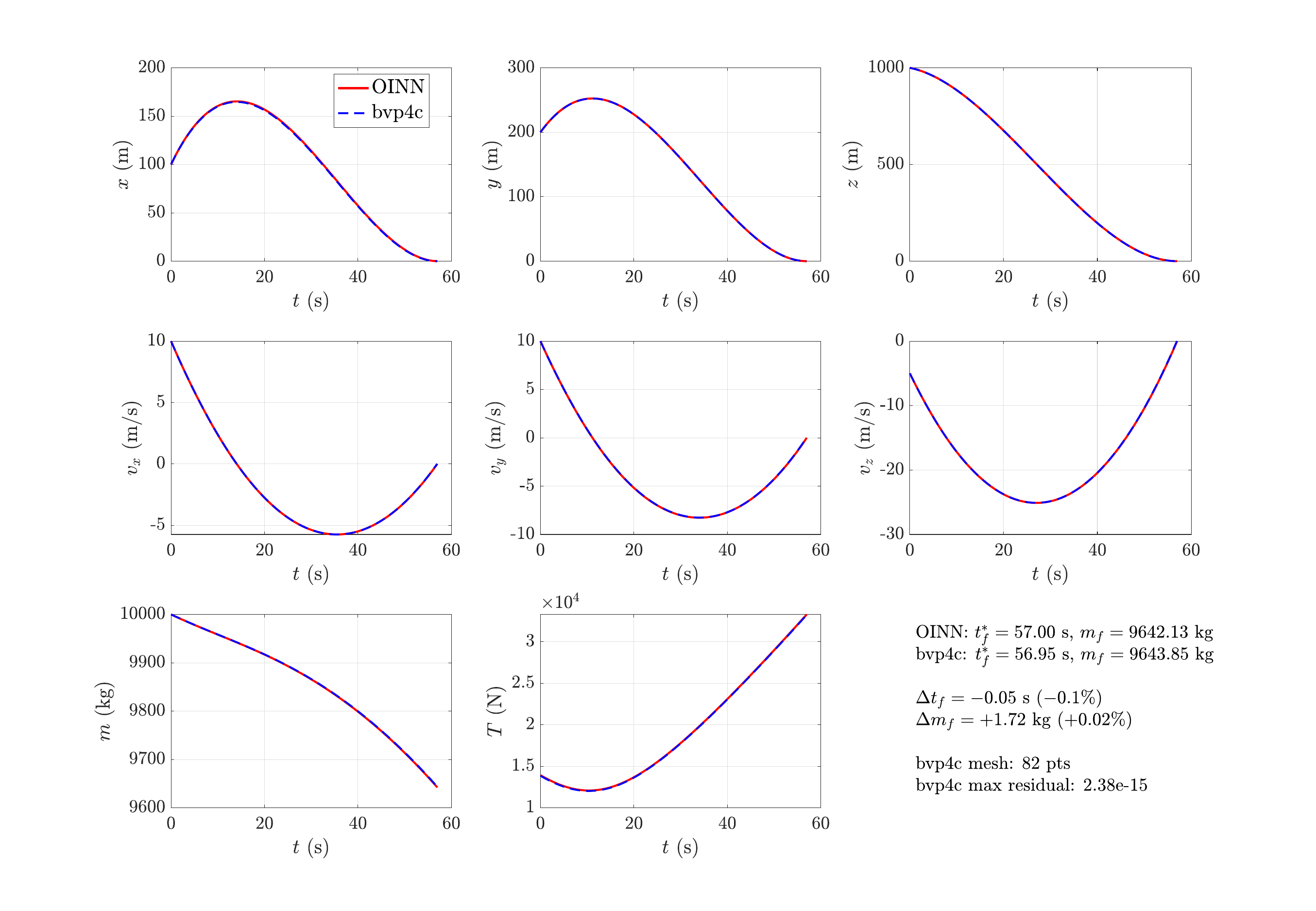}
\caption{OINN solution vs bvp4c indirect solution for T1 (baseline).}
\label{fig:t1}
\end{figure*}

\begin{figure*}
\centering
\includegraphics[width=2\columnwidth]{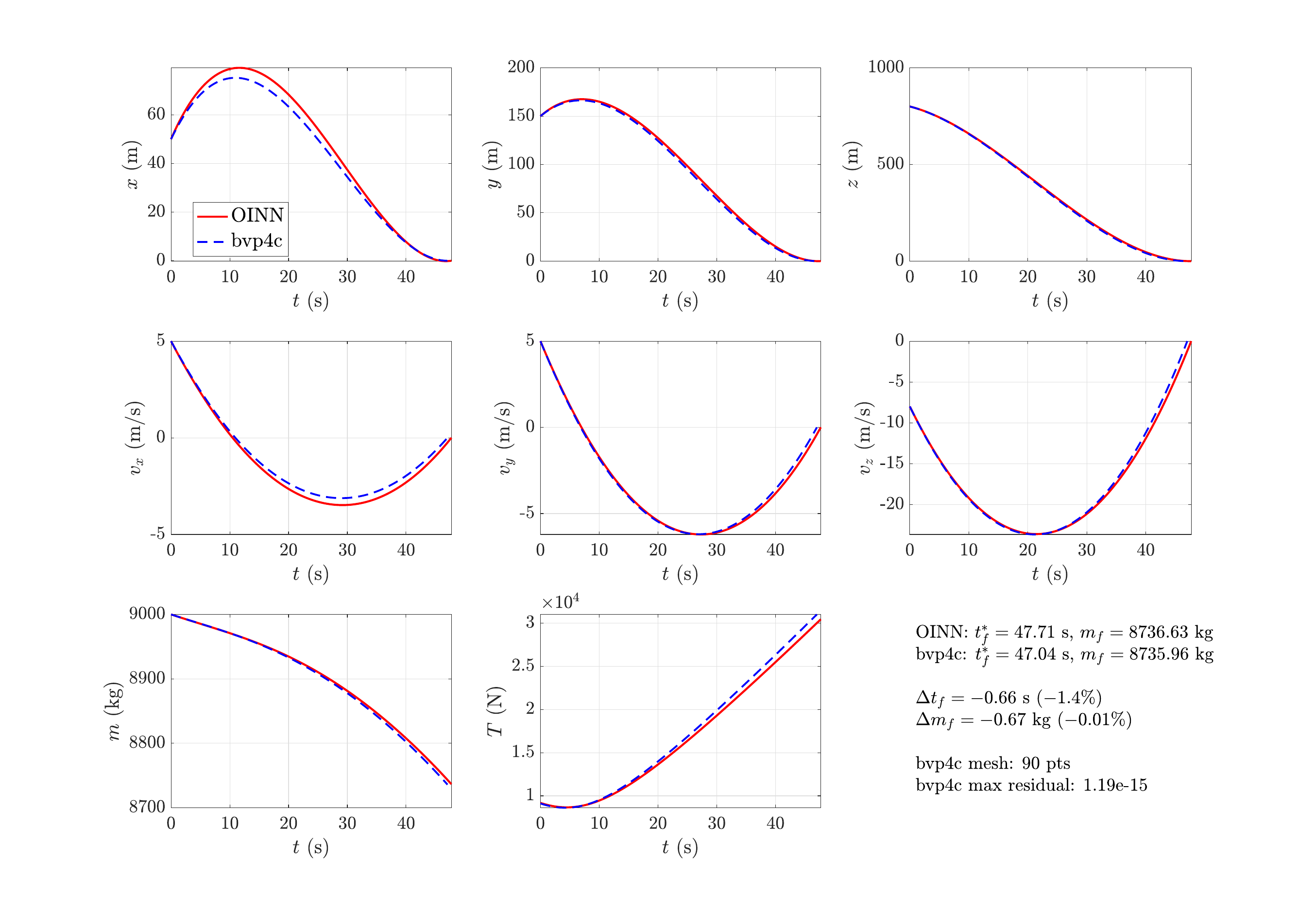}
\caption{OINN solution vs bvp4c indirect solution for T2 (light + near/slow).}
\label{fig:t2}
\end{figure*}

\begin{figure*}
\centering
\includegraphics[width=2\columnwidth]{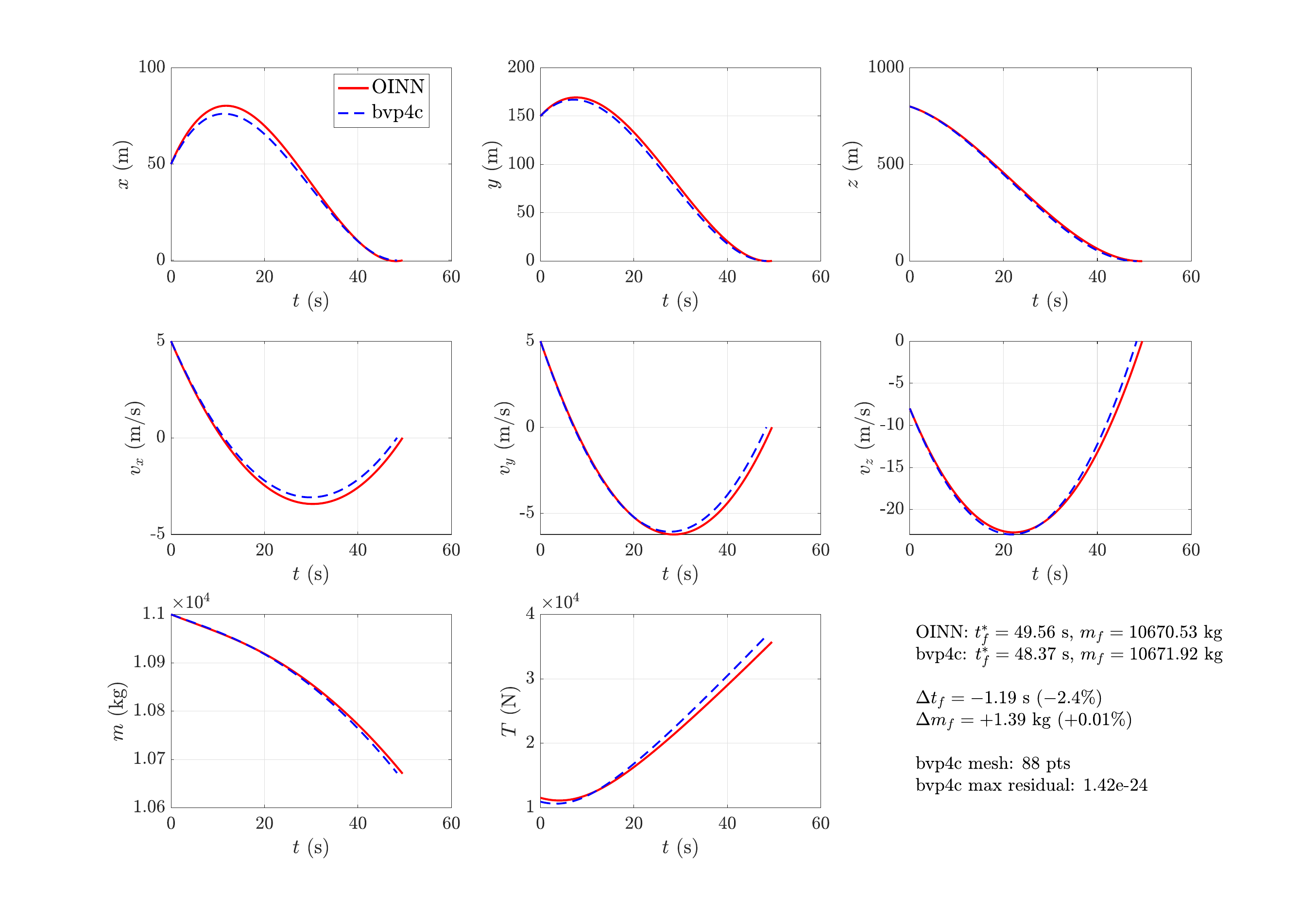}
\caption{OINN solution vs bvp4c indirect solution for T3 (heavy + near/slow).}
\label{fig:t3}
\end{figure*}

\begin{figure*}
\centering
\includegraphics[width=2\columnwidth]{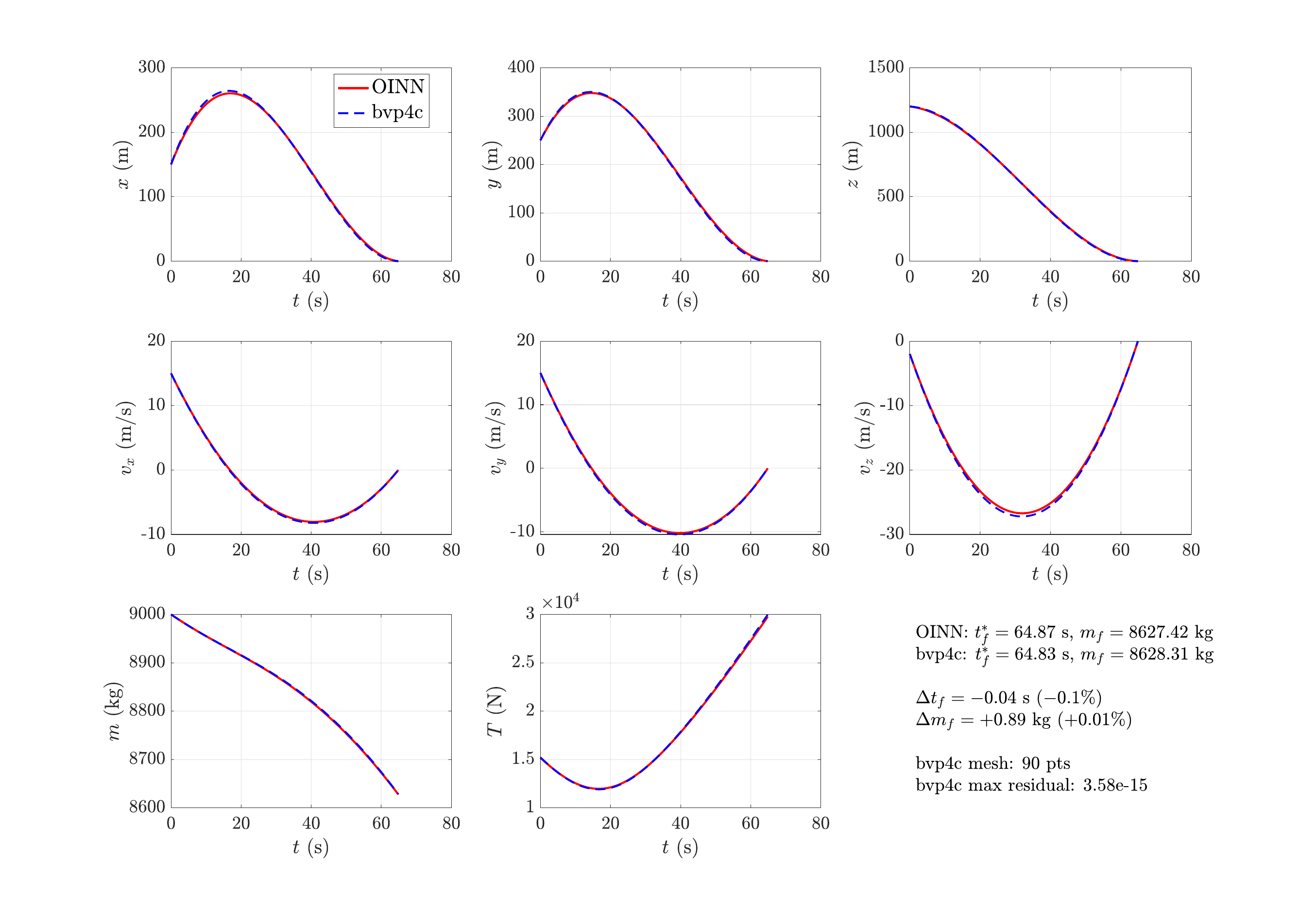}
\caption{OINN solution vs bvp4c indirect solution for T4 (light + far/fast).}
\label{fig:t4}
\end{figure*}

\begin{figure*}
\centering
\includegraphics[width=2\columnwidth]{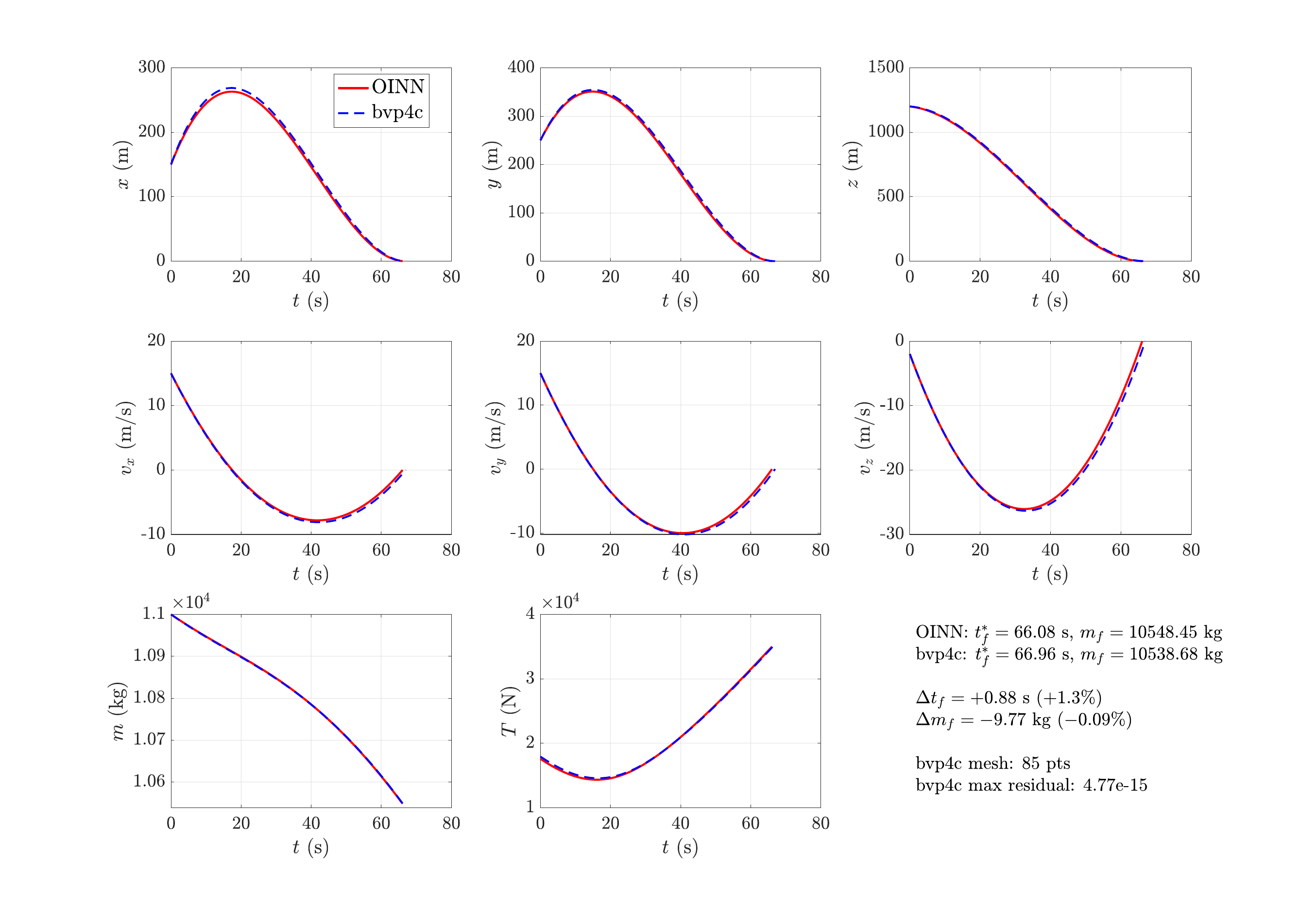}
\caption{OINN solution vs bvp4c indirect solution for T5 (heavy + far/fast).}
\label{fig:t5}
\end{figure*}

\begin{figure*}
\centering
\includegraphics[width=2\columnwidth]{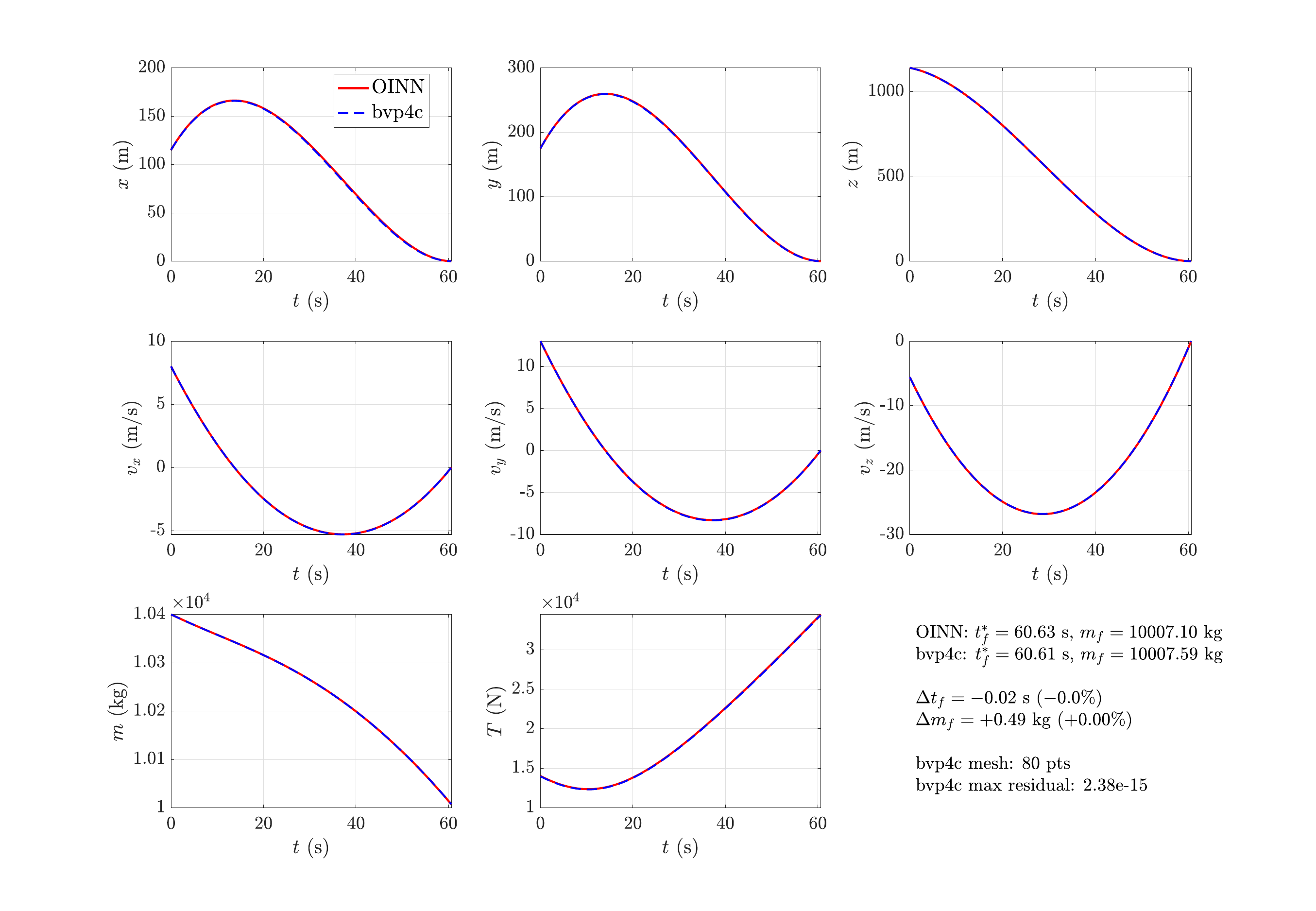}
\caption{OINN solution vs bvp4c indirect solution for T6 (interior off-center point).}
\label{fig:t6}
\end{figure*}

Across all six cases, the agreement between the trained network and the independently solved boundary-value problem is close. The predicted flight time differs by at most $1.19\ \mathrm{s}$, or $2.4\,\%$, at T3, and the predicted final mass differs by at most about ten kilograms out of several thousand, under one tenth of one percent, at T5. The \texttt{bvp4c} solver itself converges to a tight numerical tolerance at every point, with a maximum boundary-condition residual no larger than about $5\times10^{-15}$, and at one point as small as $1.4\times10^{-24}$, on a mesh of under one hundred points in every case, confirming that the indirect-method solution is itself an accurate solution of the necessary conditions of Section~\ref{sec:optimality_conditions} and therefore a meaningful independent check on the network. The differences observed are not one-sided. At some points the network predicts a slightly higher final mass than \texttt{bvp4c}, and at others a slightly lower one, consistent with both being independent, approximately converged representations of the same underlying extremal rather than one being a biased approximation of the other. These differences are exactly the empirical residual magnitudes $\varepsilon_{x}^{\infty}$ and $\varepsilon_{H}^{\infty}$ that Theorem~\ref{thm:gronwall} and Remark~\ref{rem:tf_sensitivity} translate into touchdown-error and flight-time-error bounds. The sub-percent final-mass agreement observed here is the realization, at these six points, of the small touchdown error that a small dynamics residual was shown in Section~\ref{subsec:accuracy} to guarantee, and the largest flight-time discrepancy, at T3, is consistent with a correspondingly small but nonzero free-final-time transversality residual $\varepsilon_{H}^{\infty}$ at that point.

Figures~\ref{fig:traj}--\ref{fig:training} summarize the same trained network's behavior across all six points together. Figure~\ref{fig:traj} overlays the six commanded three-dimensional trajectories, each starting from its own initial position and converging to the common target at the origin. The visibly different path shapes and lengths are produced entirely by the condition encoder's per-initial-state affine velocity-costate parameters and predicted final time, with no change to the network weights between points. Figure~\ref{fig:thrust} overlays the corresponding thrust-magnitude histories, each ending at its own predicted final time $t_{f}^{\ast}$, and Figure~\ref{fig:mass} overlays the corresponding mass-depletion histories, each starting from its own initial mass $m_{0}$ and ending at its own predicted final mass. Consistent with the energy-optimal, free-final-time formulation of Section~\ref{sec:problem_formulation}, every thrust profile in Figure~\ref{fig:thrust} is lowest in the early-to-middle portion of the flight and rises toward the end as the vehicle decelerates into the fixed terminal velocity, with the heaviest, farthest, fastest case, T5, requiring both the longest flight time and among the largest thrust magnitudes of the six. Figure~\ref{fig:training} shows the training-loss history on a logarithmic scale for the single run that produced the network evaluated in Figures~\ref{fig:t1}--\ref{fig:mass}. The total residual loss falls by several orders of magnitude within the first few thousand iterations and then settles into a noisy plateau for the remainder of the $K=25{,}000$ iterations, exactly the behavior anticipated by Theorem~\ref{thm:convergence}, whose bound~\eqref{eq:convergence_rate} guarantees convergence only to a neighborhood of a stationary point of a radius controlled by the learning rate $\eta$ and the stochastic-gradient variance of Assumption~\ref{asmp:gradient_noise}, rather than to an ever-decreasing value. The plateau visible in Figure~\ref{fig:training} is the expected signature of having reached that neighborhood, rather than indicating a failure to train further.

\begin{figure}[H]
\centering
\includegraphics[width=\columnwidth]{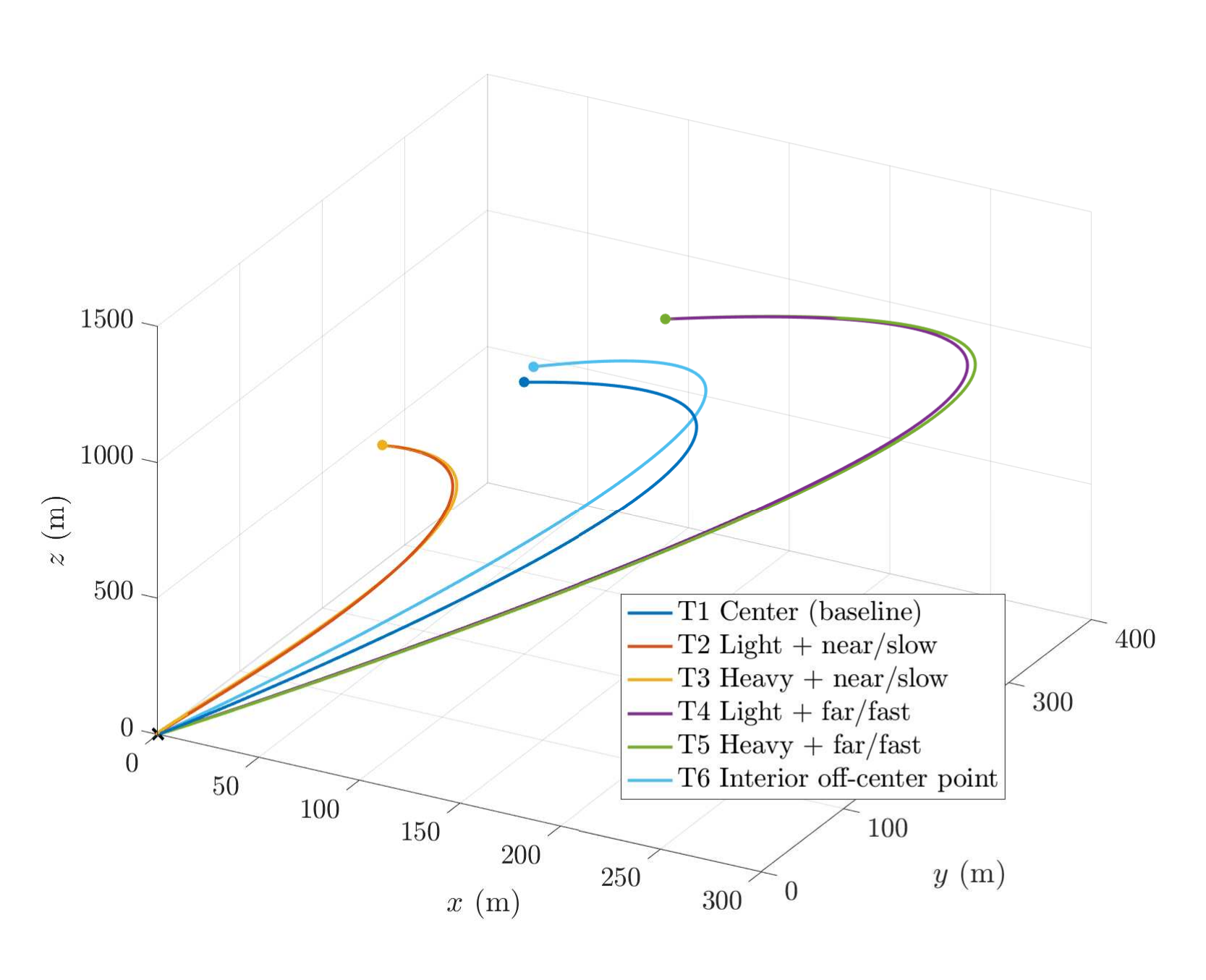}\\
\caption{Optimal three-dimensional trajectories (same network, different initial states.}
\label{fig:traj}
\end{figure}

\begin{figure}[H]
\centering
\includegraphics[width=\columnwidth]{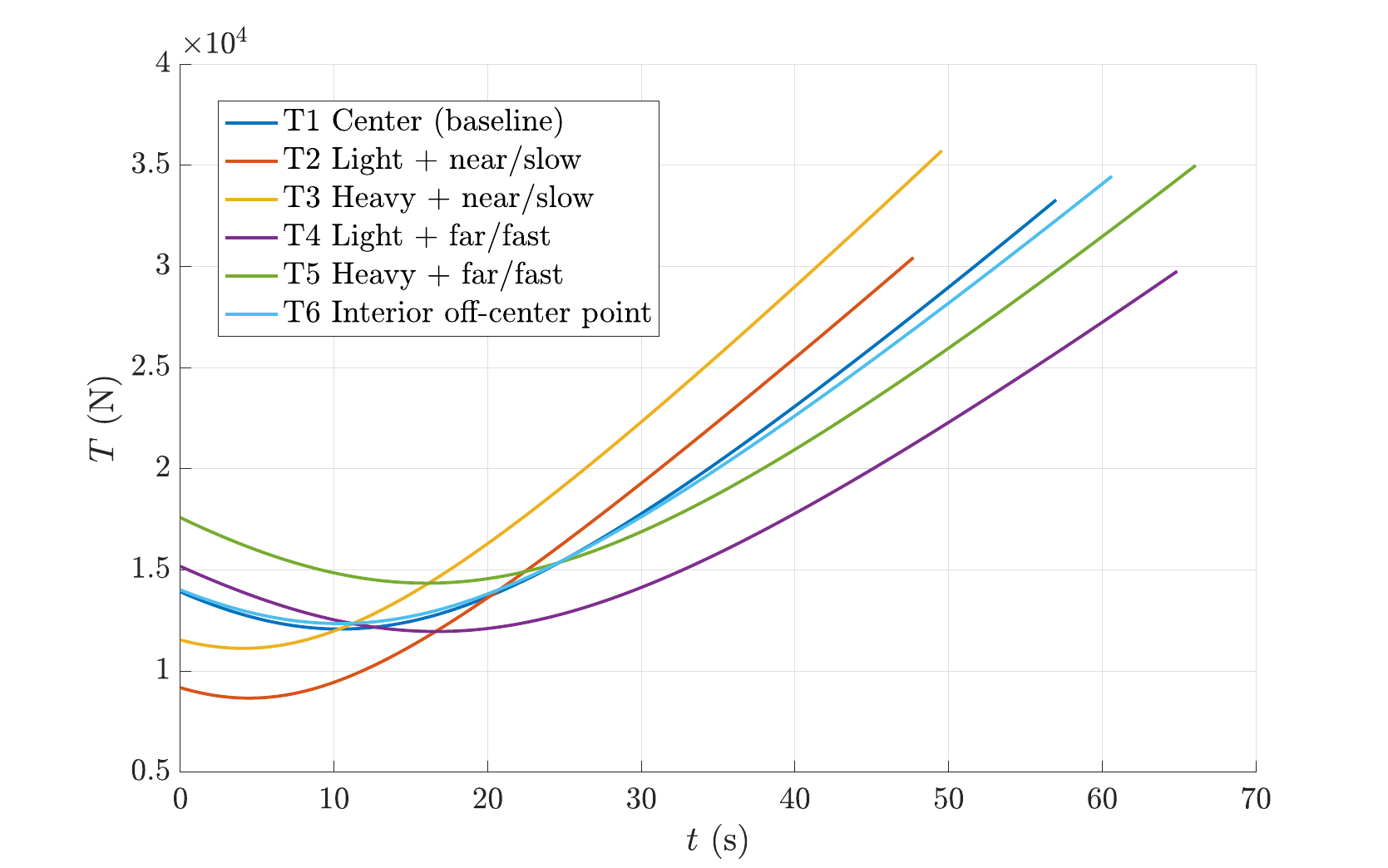}\\
\caption{Optimal thrust magnitude profiles (same network, different initial states, different $t_f^*$).}
\label{fig:thrust}
\end{figure}

\begin{figure}[H]
\centering
\includegraphics[width=\columnwidth]{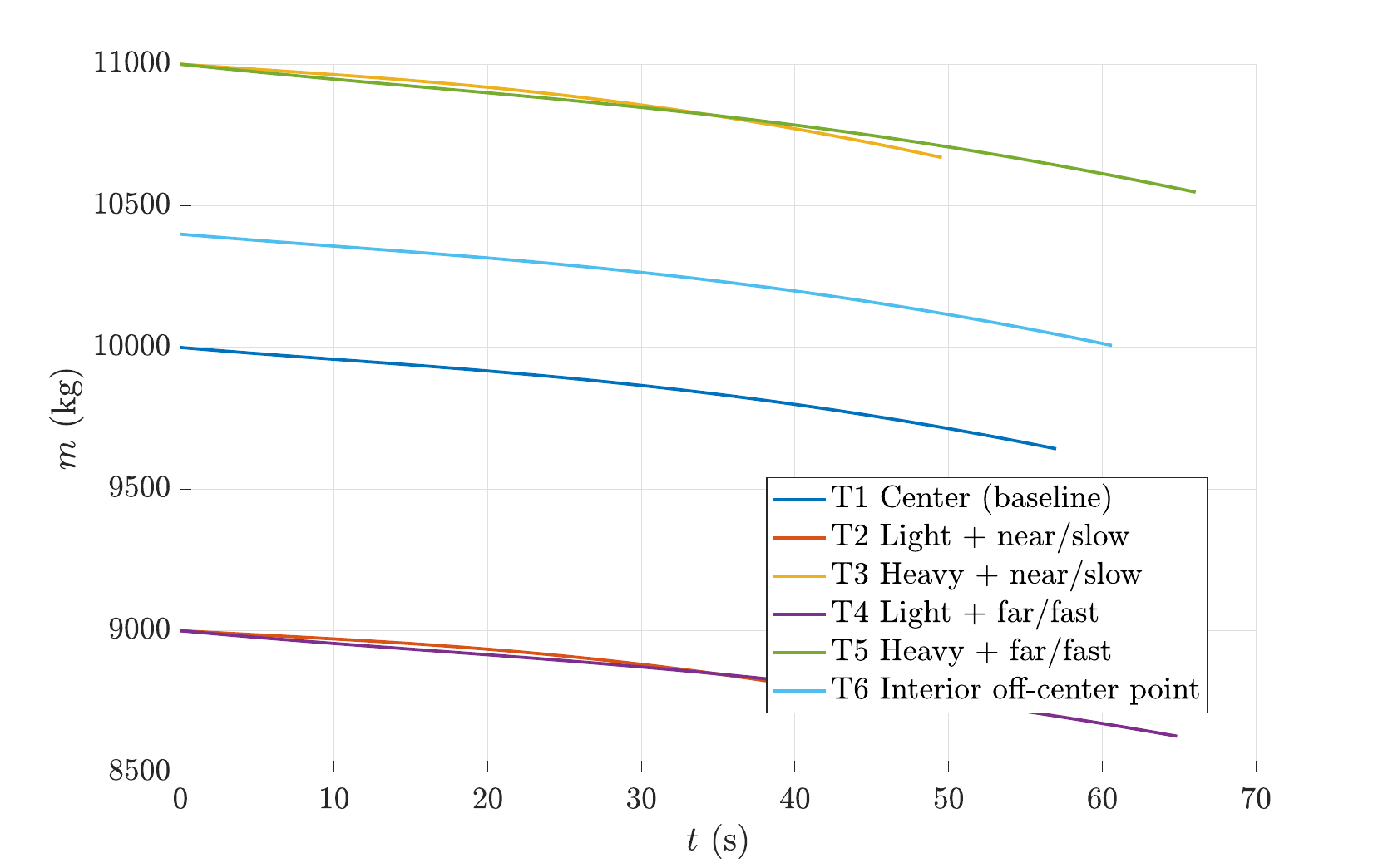}\\
\caption{Optimal mass profiles (same network, different $m_0$, different $t_f^*$).}
\label{fig:mass}
\end{figure}

\begin{figure}[H]
\centering
\includegraphics[width=\columnwidth]{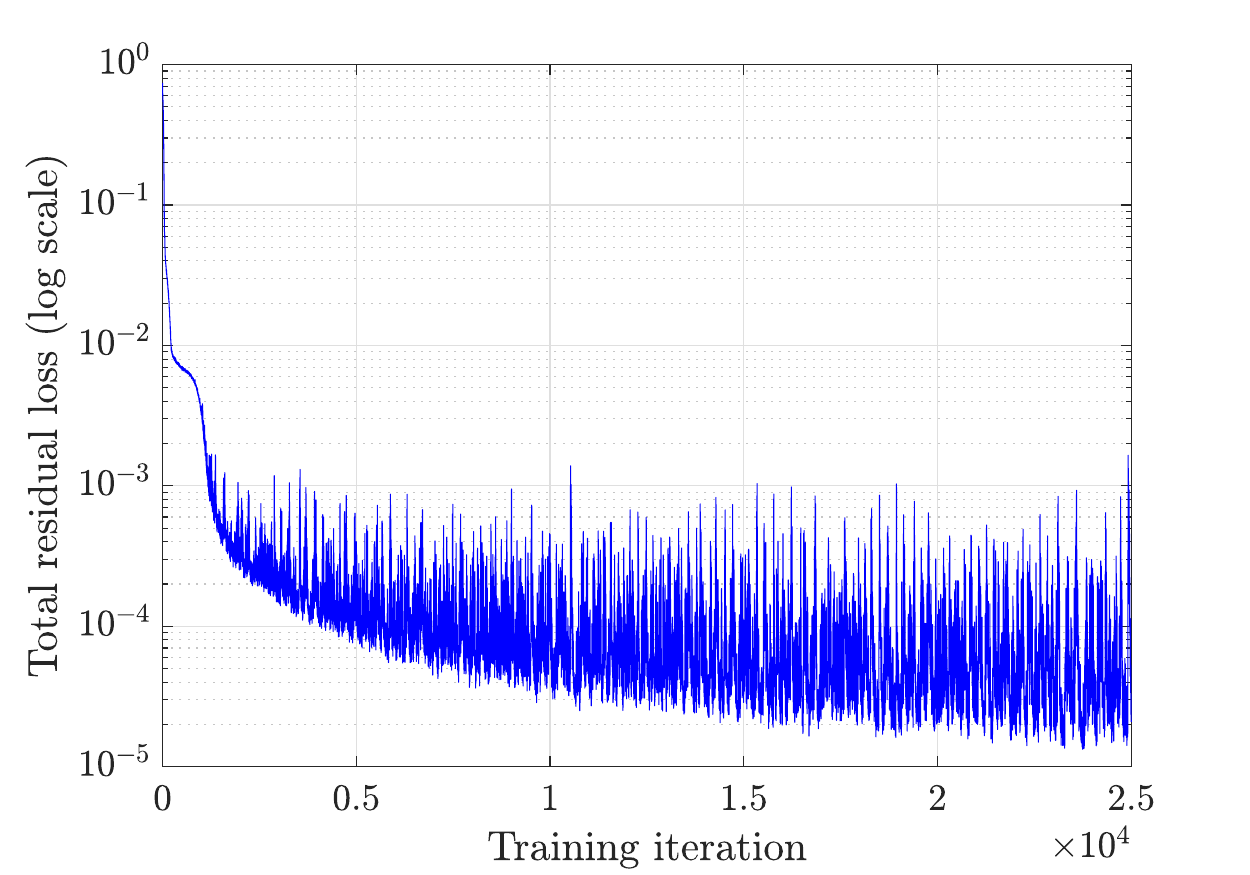}\\
\caption{Training convergence.}
\label{fig:training}
\end{figure}

\subsection{Zero-Shot Generalization via Monte Carlo Evaluation}
\label{subsec:monte_carlo}

The six cases investigated above establish generalization in one sense: a single trained network, evaluated with no retraining, reproduces an independently solved optimal solution at the center, every corner, and one interior point of $\Omega$. A complementary, statistically broader sense of generalization is also of interest for onboard use, which is regarding how the same frozen network behaves across a large number of initial states drawn at random from the same operating envelope, rather than at only a handful of named points chosen in advance. To assess this, the trained network is evaluated, again with no retraining, at eighty initial states $(\bm{r}_{0}^{(i)},\bm{v}_{0}^{(i)},m_{0}^{(i)})$ drawn uniformly at random from $\Omega$, using a random-number-generator seed different from the one used during training, so that these eighty draws constitute genuinely out-of-sample evaluation points rather than a re-sampling of the training distribution. For each draw, the same diagnostics computed at the six named points, the predicted flight time, the final mass and propellant fraction used, the peak commanded thrust, the dynamics residuals~\eqref{eq:res_r}--\eqref{eq:res_m}, and the free-final-time transversality residual~\eqref{eq:res_H}, are recorded.

Figure~\ref{fig:monte_carlo_1} shows, as box-and-strip plots over the eighty draws, the distribution of the predicted flight time $t_{f}^{\ast}$ and of the fraction of available propellant, $(m_{0}-m(t_{f}^{\ast}))/(m_{0}-m_{\mathrm{dry}})$, consumed by each trajectory. Both distributions are tightly clustered, with the flight times spanning a range consistent with the $47.7$ to $66.1\ \mathrm{s}$ range already observed across the six named points, and with no draw approaching either the minimum-flight-time safeguard $t_{f}^{\min}$ or an excessively long flight, indicating that the predicted final time~\eqref{eq:tf_softplus} responds sensibly to whichever initial state is supplied. Figure~\ref{fig:monte_carlo_2} shows the corresponding distributions of the position-, velocity-, and mass-dynamics residuals, the empirical realizations of the components of $\bm{\varepsilon}_{x}(t)$ in Theorem~\ref{thm:gronwall}. All three remain small and comparably sized across the full set of draws, with no outlier draw exhibiting a qualitatively larger residual than the rest, supporting the use of a single representative value of $\varepsilon_{x}^{\infty}$ when applying the touchdown-error bound~\eqref{eq:gronwall_bound} across the whole operating envelope rather than only at the point where it happens to be evaluated. Figure~\ref{fig:monte_carlo_3} shows the distribution of the peak commanded thrust against the engine ceiling $T_{\max}=44{,}000\ \mathrm{N}$, drawn as a reference line, together with the distribution of the maximum free-final-time transversality residual $\varepsilon_{H}^{\infty}$ of Remark~\ref{rem:tf_sensitivity}. The peak thrust remains comfortably below the ceiling for the large majority of draws, with only a small fraction approaching saturation, and the transversality residual remains small and similarly distributed across draws, consistent with the small flight-time discrepancies already observed at the six named points and with the sensitivity bound~\eqref{eq:tf_sensitivity} translating that small residual into a correspondingly small flight-time error throughout $\Omega$, not merely at the points checked individually against \texttt{bvp4c}.

\begin{figure}[H]
\centering
\includegraphics[width=\columnwidth]{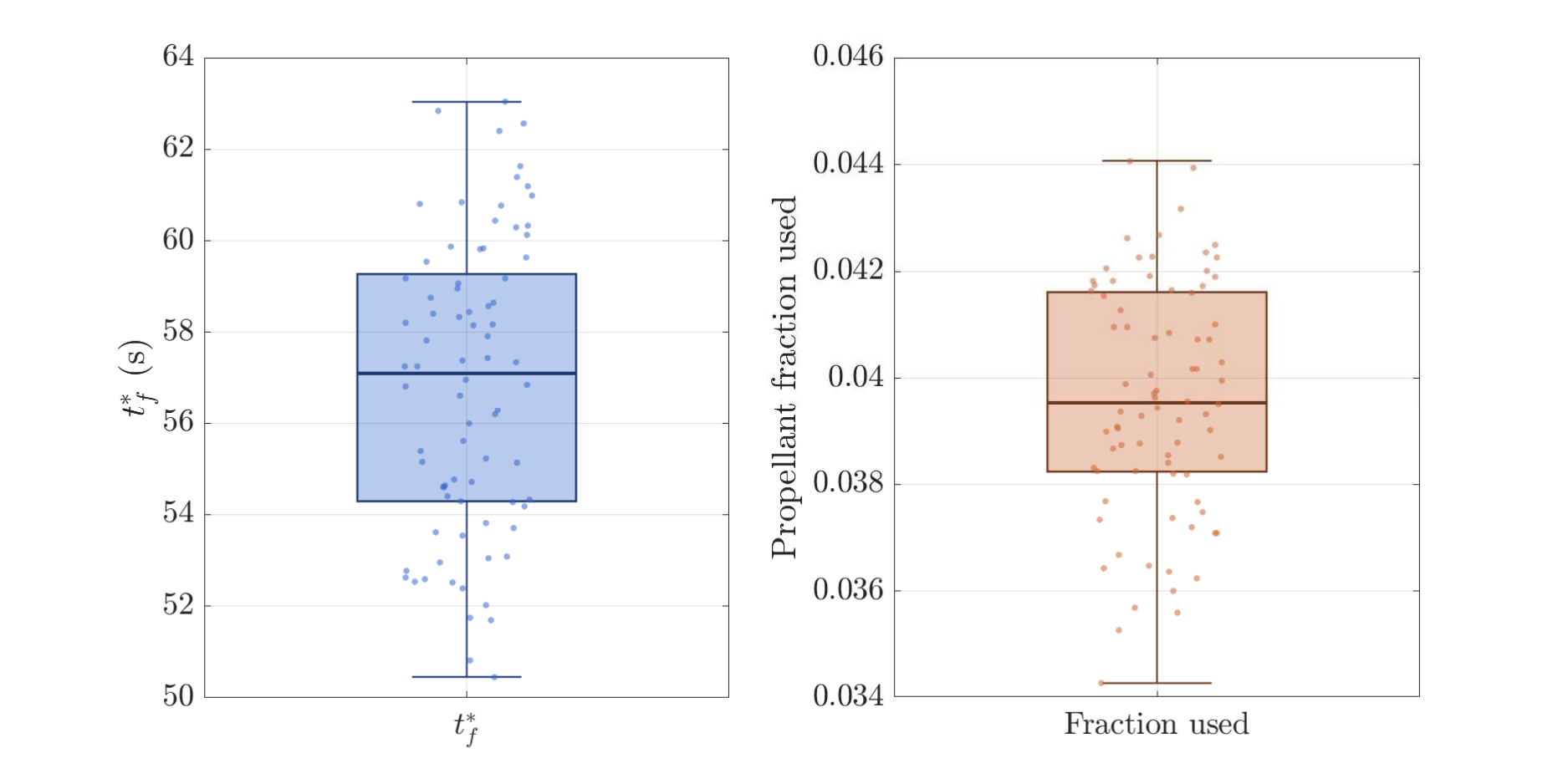}\\
\caption{Monte Carlo: predicted flight time and propellant usage.}
\label{fig:monte_carlo_1}
\end{figure}

\begin{figure}[H]
\centering
\includegraphics[width=\columnwidth]{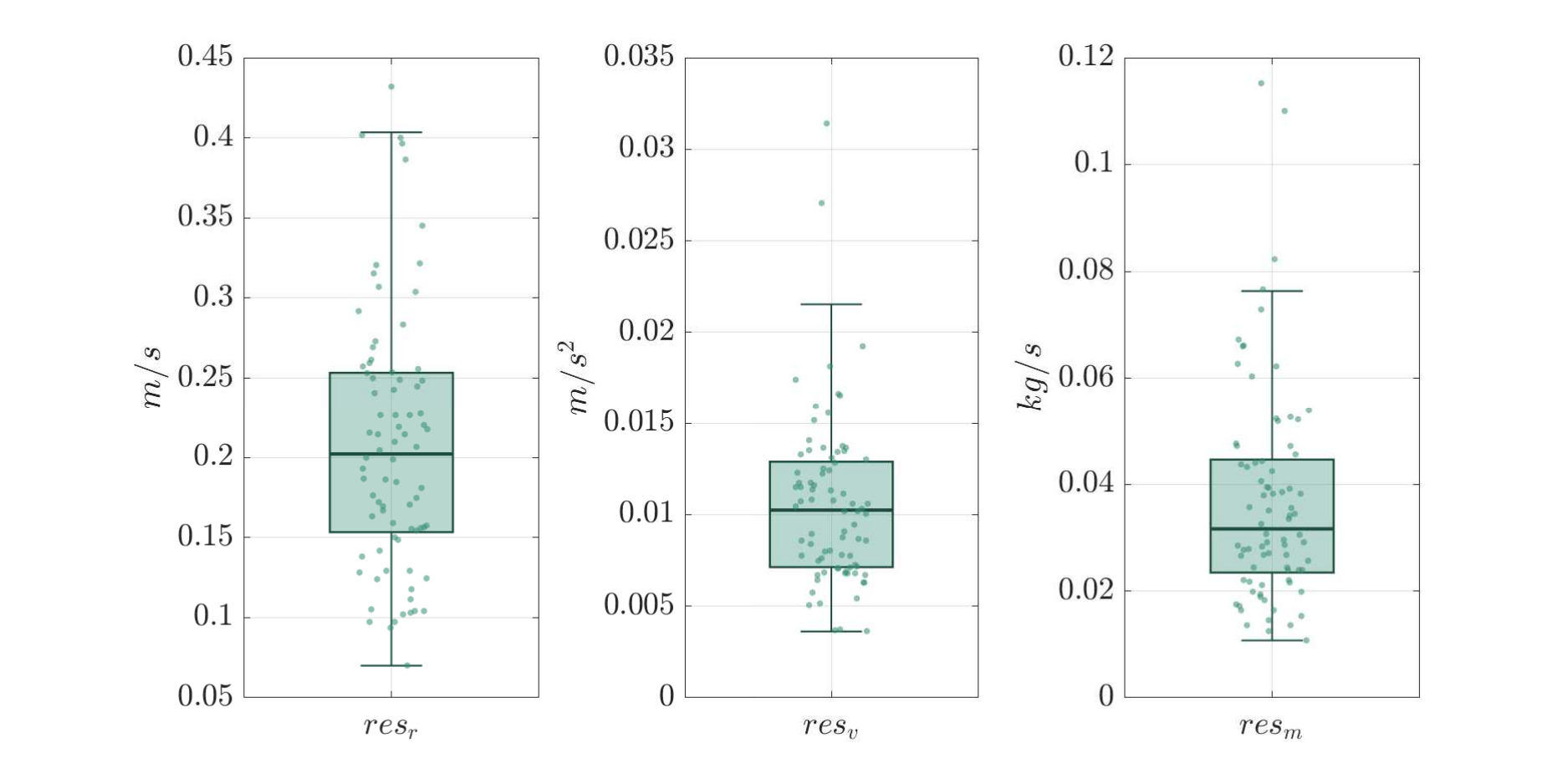}\\
\caption{Monte Carlo: position-dynamics, velocity-dynamics, and mass-dynamics residuals.}
\label{fig:monte_carlo_2}
\end{figure}

\begin{figure}[H]
\centering
\includegraphics[width=\columnwidth]{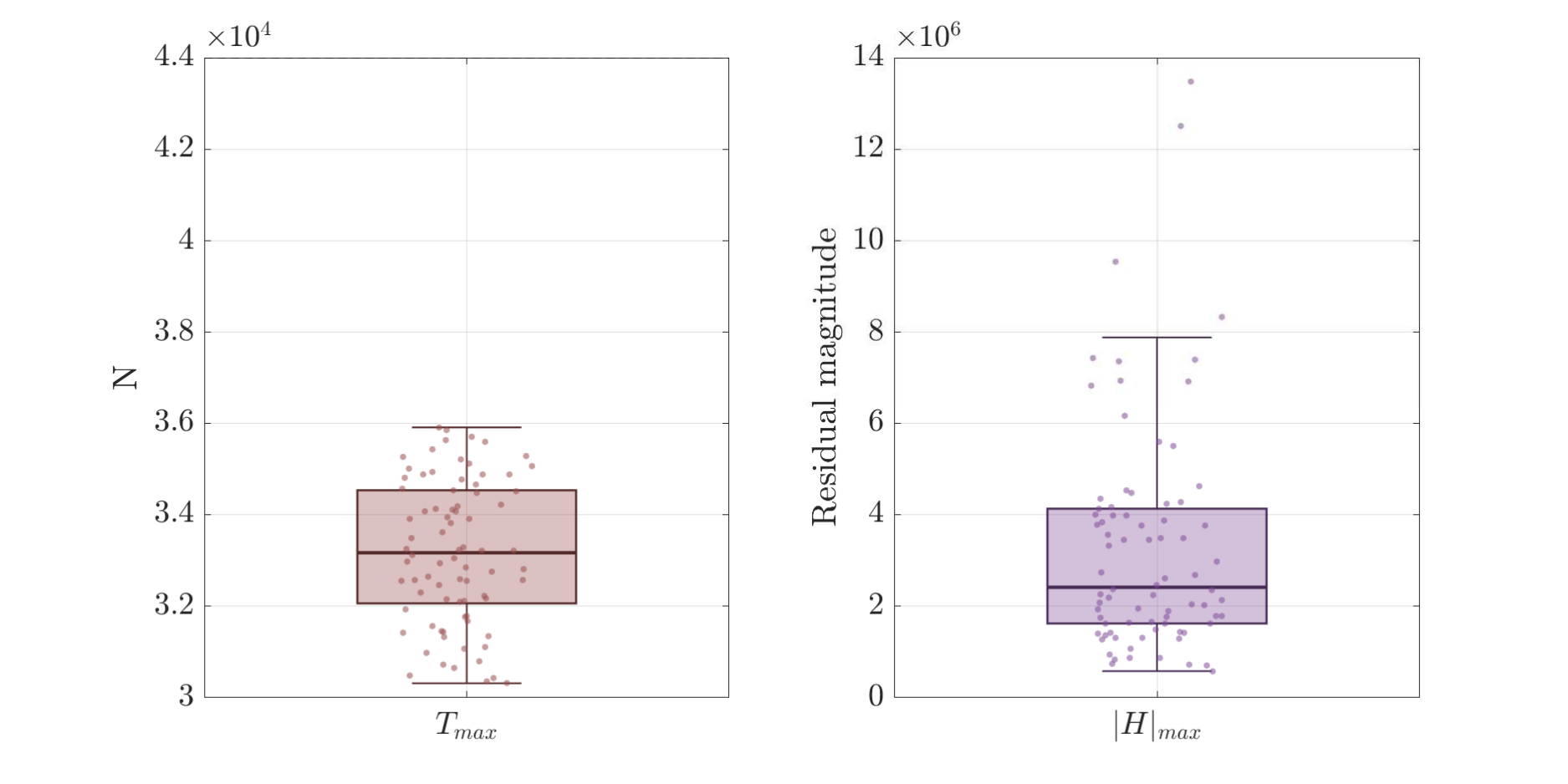}\\
\caption{Monte Carlo: thrust range and Hamiltonian check.}
\label{fig:monte_carlo_3}
\end{figure}

Together, the six-point comparison against the independently solved indirect method and the eighty-point Monte Carlo evaluation give complementary evidence for the theoretical claims in Section~\ref{sec:theoretical_analysis}. The trained network, evaluated zero-shot and without retraining anywhere in $\Omega$, reproduces an independently computed numerical solution of the necessary conditions of optimality closely at six representative points, including the hardest corner of the operating envelope, and maintains comparably small dynamics and transversality residuals across eighty additional, randomly drawn points spanning the same envelope. By the bounds of Theorem~\ref{thm:gronwall} and Remark~\ref{rem:tf_sensitivity}, these small, consistently sized residuals translate directly into small, bounded touchdown-position, touchdown-velocity, and flight-time errors throughout $\Omega$, while the per-update computational and memory cost of evaluating the policy, quantified in Section~\ref{subsec:onboard_cost}, remains the same small, fixed value at every one of these points.

\section{Conclusions}
\label{sec:conclusions}

This paper developed an Optimality-Informed Neural Network (OINN) approach for the energy-optimal, free-final-time powered descent of a lunar lander from any initial position, velocity, and mass within a bounded operating envelope to a fixed landing site. Specializing a recently proposed optimality-principles-informed learning framework to this free-final-time, fixed-terminal-boundary structure, the approach hard-encodes every boundary and transversality condition identified by Pontryagin's minimum principle directly into the network architecture, substitutes the closed-form optimal thrust and direction law rather than learning it, and trains the remaining state, costate, and auxiliary value-function outputs against a physics-residual loss built entirely from the necessary conditions of optimality, without requiring any precomputed optimal trajectories. A companion theoretical analysis showed that the offline training procedure converges, in a stationarity sense standard for nonconvex stochastic optimization, to a neighborhood of a point satisfying these conditions; that several of the most safety-critical conditions, including the initial and terminal boundary conditions and the control admissibility constraints, hold exactly regardless of training quality; that the achieved training residual translates, through an explicit Gr\"{o}nwall-type bound, into a computable bound on touchdown position, touchdown velocity, and flight-time error; and that the trained policy requires only a fixed, small number of floating-point operations and a memory footprint of a few tens of kilobytes to evaluate online, independent of the specific initial state being flown.
Numerical simulations of a representative lunar landing scenario supported these theoretical properties. The trained policy, evaluated with no retraining, agreed closely with an independently solved indirect-method boundary-value-problem solution at six representative initial states spanning the operating envelope, with flight-time and final-mass differences of at most a few percent and a small fraction of one percent, respectively. An eighty-point Monte Carlo evaluation over the same operating envelope showed comparably small dynamics and transversality residuals throughout, with no evidence of thrust saturation or degraded behavior away from the six selected points, supporting the claim that a single offline-trained network, requiring no expert trajectory dataset and no per-instance online numerical optimization, can serve as a physically consistent, real-time-deployable guidance law for the entire family of powered-descent initial conditions considered.


\acknowledgements 
Generative AI tools were used for language refinement during the preparation of this paper. All technical content and interpretations were developed by the author.

\bibliographystyle{IEEEtran}
\bibliography{MyBibFile}

\thebiography
\begin{biographywithpic}
{Zhenbo Wang}{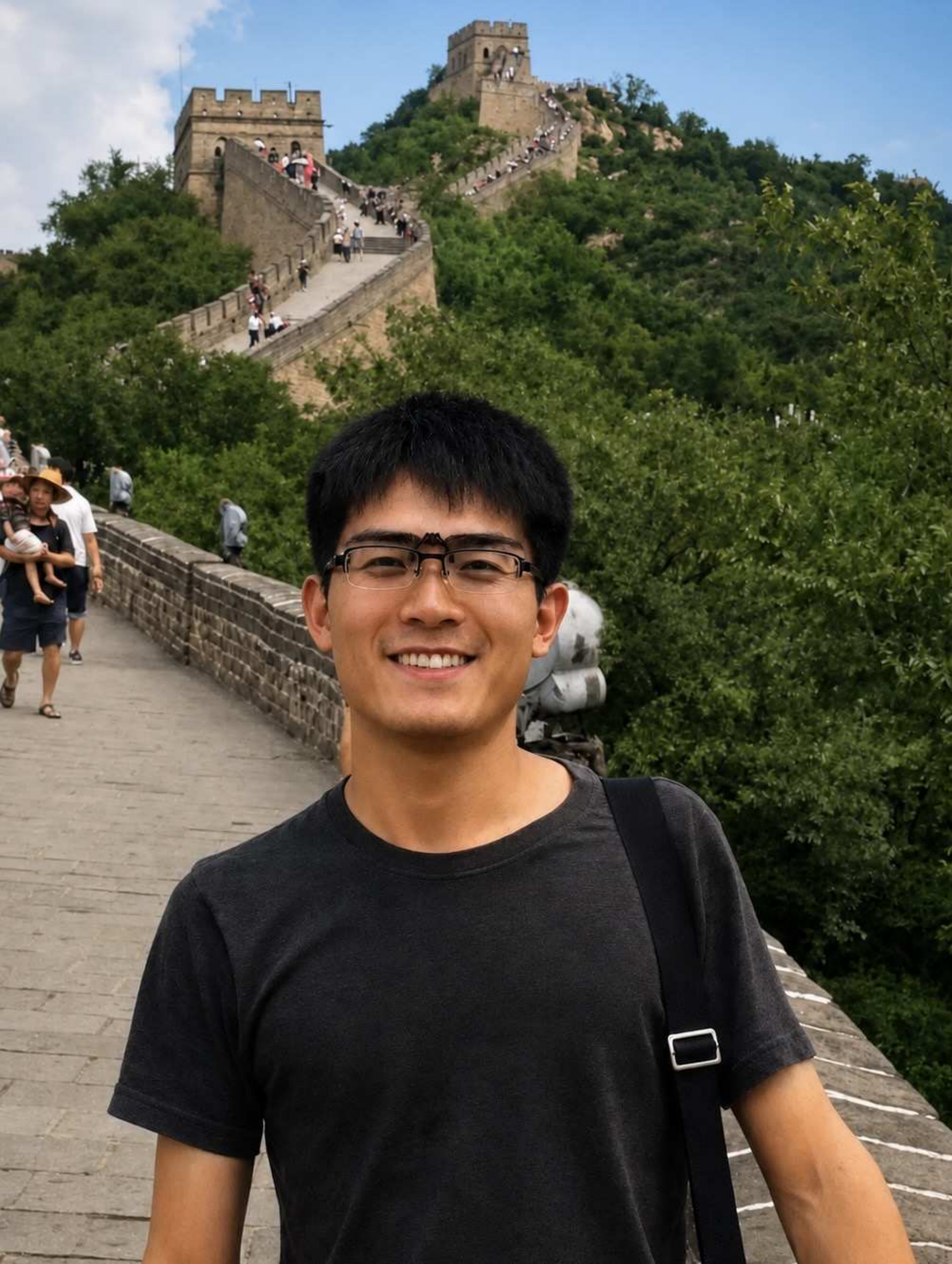} received his B.E. degree in Astronautics from Nanjing University of Aeronautics and Astronautics in 2010 and his M.E. degree in Control Engineering from Beihang University in 2013. In 2018, he received his Ph.D. degree in Aeronautics and Astronautics from Purdue University and joined the University of Tennessee Knoxville (UTK) as an Assistant Professor. He is now an Associate Professor in the Department of Mechanical and Aerospace Engineering and the director of the Autonomous Systems Laboratory at UTK. He is a recipient of the 2023 NSF Faculty Early Career Development Program (CAREER) Award, the 2023 Louis and Ann Hoffman Endowed Excellence in Research Award, and the 2024 Professional Promise in Research Award. His research interests are control, optimization, and machine learning for various engineering applications including space systems, air vehicles, connected and automated vehicles, and power and energy systems. He is a Senior Member of the American Institute of Aeronautics and Astronautics (AIAA) and a member of the AIAA Atmospheric Flight Mechanics (AFM) Technical Committee. He is a Senior Editor of \textit{IEEE Transactions on Aerospace and Electronic Systems}.
\end{biographywithpic}

\end{document}